\documentclass{lmsArXiv}

\title{\textsc{$p$-adic $L$-functions on metaplectic groups}}
\author{Salvatore Mercuri}

\usepackage{amssymb, amsmath, enumerate, mathrsfs, tikz, wasysym, hyperref, comment, graphicx, mathtools}
\usepackage[margin=1.2in]{geometry}
\usetikzlibrary{calc, decorations.markings}
\usetikzlibrary{positioning}
\usetikzlibrary{matrix}

\renewcommand\Re{\operatorname{\mathfrak{Re}}}
\renewcommand\Im{\operatorname{\mathfrak{Im}}}

\newcommand{\eps}{\varepsilon}

\newcommand{\spa}{\operatorname{span}}

\newcommand{\id}{\operatorname{id}}

\newcommand{\tr}{\operatorname{tr}}
\newcommand{\Hom}{\operatorname{Hom}}

\newcommand{\ph}{\operatorname{\varphi}}
\newcommand{\mfrak}{\mathfrak}
\newcommand{\back}{\backslash}

\renewenvironment{psmallmatrix}
  {\left(\begin{smallmatrix}}
  {\end{smallmatrix}\right)}
\newcommand{\AAQ}{\operatorname{\mathbb{A}_{\mathbb{Q}}}}

\newcommand{\Tr}{\mathop{\rm Tr}\nolimits}
\newcommand{\pr}{\mathop{\rm pr}\nolimits}
\newcommand{\sgn}{\operatorname{sgn}}
\newcommand{\Vol}{\operatorname{Vol}}

\newcommand{\ord}{\operatorname{ord}}

\newcommand{\tors}{\operatorname{tors}}

\newcommand{\diag}{\operatorname{diag}}
\newcommand{\action}[1]{{}^{\sigma}\!{{#1}}}
\newcommand{\Waction}[1]{{}^{\iota_{\chi}^{-1}}\!{{#1}}}
\newcommand{\iotaaction}[1]{{}^{\iota}\!{{#1}}}
\newcommand*{\myfont}{\fontfamily{pag}\selectfont}
\newcommand\Sx{\text{{\myfont X}}}
\newcommand\So{\text{{\myfont O}}}

\newtheorem{lemma}{Lemma}[section]
\newtheorem{theorem}[lemma]{Theorem}
\newtheorem{corollary}[lemma]{Corollary}
\newtheorem{proposition}[lemma]{Proposition}

\newnumbered{assertion}{Assertion}    
\newnumbered{conjecture}{Conjecture}  
\newnumbered{definition}{Definition}
\newnumbered{hypothesis}{Hypothesis}
\newnumbered{remark}[lemma]{Remark}
\newnumbered{note}{Note}
\newnumbered{observation}{Observation}
\newnumbered{problem}{Problem}
\newnumbered{question}{Question}
\newnumbered{algorithm}{Algorithm}
\newnumbered{example}{Example}
\newunnumbered{notation}{Notation}

\numberwithin{equation}{section}

\classno{11F46, 11M41 (primary)}

\extraline{Completed with the support of an EPSRC Doctoral Studentship}

\begin{document}
\maketitle 

\begin{abstract} 
\noindent With respect to the analytic-algebraic dichotomy, the theory of Siegel modular forms of half-integral weight is lopsided; the analytic theory is strong whereas the algebraic lags behind. In this paper, we capitalise on this to establish the fundamental object needed for the analytic side of the Iwasawa main conjecture -- the $p$-adic $L$-function obtained by interpolating the complex $L$-function at special values. This is achieved through the Rankin-Selberg method and the explicit Fourier expansion of non-holomorphic Siegel Eisenstein series. The construction of the $p$-stabilisation in this setting is also of independent interest.
\end{abstract}

\tableofcontents

\section{Introduction}
\label{intro}

\noindent Traditionally, $p$-adic $L$-functions have dual constructions -- analytic and algebraic -- and it is the substance of the Iwasawa main conjecture that these two are equivalent. This conjecture can be formulated for various settings; for example, over $GL_1$, the conjecture asserts that the analytic construction -- Kubota-Leopoldt's $p$-adic interpolation of the Dirichlet $L$-function -- is equivalent to Iwasawa's algebraic $p$-adic $L$-function. The Iwasawa main conjecture for classical modular forms of integral weight is formulated over $GL_2$ and this has been a recent, active research area with its connections to the Birch and Swinnerton-Dyer conjecture, see \cite{Skinner1} and \cite{Skinner}. Provided one has both the analytic and algebraic machinery, the Iwasawa main conjecture can be formulated for higher dimensional modular forms and groups, for example \cite{Wan}

The algebraic theory of half-integral weight modular forms, both classical and metaplectic, has long been inchoate due to the difficulties present in developing the `Galois side'. Recent work by M. Weissmann in \cite{Weissmann} has made progress in this regard by developing $L$-groups for metaplectic covers, the length and methods of which further underline the difficulties present here. The analytic theory is substantial however, and in this paper we give the analytic construction of the $p$-adic $L$-function for Siegel modular forms of half-integral weight and any degree $n$. In \cite{Mercuri} we gave a similar construction when $n = 1$; in that case the $p$-adic $L$-function was already known to exist by the Shimura correspondence, this is not so for general $n>1$.

The proof found here is adapted from the method of A. Panchishkin found in Chapters 2 and 3 of \cite{Panchishkin}, which proves the existence of the analytic $p$-adic $L$-function for Siegel modular forms of integral weight and even degree $n$. This method makes critical use of the Rankin-Selberg method and reduces the question of $p$-adic boundedness of the $L$-function down to that of the Fourier coefficients of the Eisenstein series that are involved in the Rankin-Selberg integral expression. For full generality, it is assumed that $p$ does not divide the level of the modular form $f$ and a crucial step is to produce another form $f_0$ such that $p$ does divide the level. Significant modifications to the method of \cite{Panchishkin} were required to make this work in the metaplectic case -- this is Section \ref{pstabsection}. Outside of this, the success of Panchishkin's method is facilitated by the work of G. Shimura in developing the Rankin-Selberg integral expression in this setting, \cite{Shimuraexp}, and the arithmeticity of Eisenstein series, \cite[Chapters 16--17]{Shimurabook}. Interestingly, the $p$-adic boundedness of the Eisenstein series coefficients is almost immediate in this case, making the final step of the proof simpler than that found in \cite{Panchishkin}.

After preliminary Sections \ref{modformsection} and \ref{Lfunsection}, we establish the $p$-stabilisation in Section \ref{pstabsection}. Fairly elementary manipulations on the level of the Rankin-Selberg integral follow in Section \ref{tracesection}. Sections \ref{thetasection} and \ref{eisensteinsection} are devoted to transformation formulae of theta series and Fourier expansions of Eisenstein series -- these are relatively well-known. Finally the statement and proof of the main theorem, and the subsequent existence of the $p$-adic $L$-function, are given in Section \ref{interpolation}.

\section{Siegel modular forms}
\label{modformsection}

\noindent This section runs through the very basics of the modular forms that we study and their Fourier expansions are detailed.

For any ring $R$ and any matrix $a\in M_n(R)$, note the use of the following notation: $a>0$ ($a\geq 0$) to mean that $a$ is positive definite (resp. positive semi-definite), $|a| := \det(a)$, $\|a\| :=|\det(a)|$, and $\tilde{a} := (a^T)^{-1}$. For any collection $a_1, \dots, a_r$ of matrices with entries in $R$, let $\diag[a_1, \dots, a_r]$ be the matrix whose $j$th diagonal block is $a_j$ and is zero off the diagonal. 

Let $\AAQ$ and $\mathbb{I}_{\mathbb{Q}}$ denote the adele ring and idele group, respectively, of $\mathbb{Q}$. The Archimedean place is denoted by $\infty$ and the non-Archimedean places by $\mathbf{f}$. If $G$ is an algebraic group, let $G_{\mathbb{A}}$ denote its adelisation. Let $G_{\infty} := G(\mathbb{R})$, $G_p := G(\mathbb{Q}_p)$, and denote by $G_{\mathbf{f}}$ the subgroup of elements of $G_{\mathbb{A}}$ whose Archimedean place is the identity of $G_{\infty}$. View $G$ as a subgroup of $G_{\mathbb{A}}$ by embedding diagonally at every place, but view $G_{\infty}$ and $G_p$ as subgroups by embedding place-wise. Recall the adelic norm
\[
	|x|_{\mathbb{A}} = |x_{\infty}|\prod_p|x_p|_p
\]
where $x\in\AAQ$, $|\cdot|$ denotes the usual absolute value on $\mathbb{R}$, and $|\cdot|_p$ denotes the $p$-adic absolute value, normalised in the sense that $|p|_p = p^{-1}$. Let $\mathbb{T}$ denote the unit circle and define three $\mathbb{T}$-valued characters on $\mathbb{C}$, $\mathbb{Q}_p$, and $\AAQ$ respectively by
\begin{align*}
	e:z&\mapsto e^{2\pi iz}, \\
	e_p:x&\mapsto e(-\{x\}), \\
	e_{\mathbb{A}}:x&\mapsto e(x_{\infty})\prod_p e_p(x_p),
\end{align*}
where $\{x\}$ denotes the fractional part of $x\in\mathbb{Q}_p$; if $x\in \AAQ$ and $z\in\mathbb{C}$ then write $e_{\infty}(x) = e(x_{\infty})$ and $e_{\infty}(z) = e(z)$.

For any fractional ideal $\mfrak{r}$ of $\mathbb{Q}$ let $\mfrak{r}_p$ denote the completion (with respect to the $p$-adic norm) of the localisation of $\mfrak{r}$ at the prime $p$, which is an ideal of $\mathbb{Z}_p$. Understand $0\leq N(\mfrak{r})\in\mathbb{Q}$ to be the unique positive generator of $\mfrak{r}$.

Write any $\alpha\in GL_{2n}(\mathbb{Q})$ as
\[	
	\alpha = \begin{pmatrix} a_{\alpha} & b_{\alpha} \\ c_{\alpha} & d_{\alpha}\end{pmatrix}
\]
where $x_{\alpha}\in M_n(\mathbb{Q})$ for $x\in\{a, b, c, d\}$. Define an algebraic group $G$, subgroup $P\leq G$, and the Siegel upper half-space $\mathbb{H}_n$ by
\begin{align*}
	G:&= Sp_n(\mathbb{Q}) = \{\alpha\in GL_{2n}(\mathbb{Q})\mid \alpha^T\iota\alpha = \iota\},\ \hspace{20pt} \iota := \begin{pmatrix} 0 & -I_n \\ I_n & 0\end{pmatrix}, \\
	P:&= \{\alpha\in G\mid c_{\alpha} = 0\}, \\
	\mathbb{H}_n :&= \{z = x+iy\in M_n(\mathbb{C})\mid x, y\in M_n(\mathbb{R}), z^T = z, y>0\}.
\end{align*}

A half-integral weight is an element $k\in\mathbb{Q}$ such that $k-\frac{1}{2}\in\mathbb{Z}$; an integral weight is an element $\ell\in\mathbb{Z}$. The factor of automorphy of half-integral weight involves taking a square root; to guarantee consistency of the choice of root one uses the double metaplectic cover $Mp_n$ of $Sp_n$. The localisations $M_p := Mp_n(\mathbb{Q}_p)$ and the adelisation $M_{\mathbb{A}}$ of $Mp_n(\mathbb{Q})$ can be described as groups of unitary transformations, respectively, on $L^2(\mathbb{Q}_p^n)$ and $L^2(\AAQ^n)$ with the exact sequences
\begin{align*}
	1\to\mathbb{T}\to M_x\to G_x\to 1,
\end{align*}
where $x\in\{p, \mathbb{A}\}$. There are natural projections $\pr_{\mathbb{A}}:M_{\mathbb{A}}\to G_{\mathbb{A}}$ and $\pr_p:M_p\to G_p$, either of which will usually be denoted $\pr$ as the context is clear. On the flip side, there are natural lifts $r:G\to M_{\mathbb{A}}$ and $r_P:P_{\mathbb{A}}\to M_{\mathbb{A}}$ through which we view $G$ and $P_{\mathbb{A}}$ as subgroups of $M_{\mathbb{A}}$.

For any two fractional ideals $\mfrak{x}, \mfrak{y}$ of $\mathbb{Q}$ such that $\mfrak{xy}\subseteq\mathbb{Z}$, congruence subgroups are defined by the following respective subgroups of $G_p, G_{\mathbb{A}}$, and $G$:
\begin{align*}
	D_p[\mfrak{x}, \mfrak{y}] :&= \{x\in G_p\mid a_x, d_x\in M_n(\mathbb{Z}_p), b_x\in M_n(\mfrak{x}_p), c_x\in M_n(\mfrak{y}_p)\}, \\
	D[\mfrak{x}, \mfrak{y}] :&= Sp_n(\mathbb{R})\prod_p D_p[\mfrak{x}, \mfrak{y}], \\
	\Gamma[\mfrak{x}, \mfrak{y}] :&= G\cap D[\mfrak{x}, \mfrak{y}].
\end{align*}
Typically these will take the form $\Gamma[\mfrak{b}^{-1}, \mfrak{bc}]$ for certain fractional ideals $\mfrak{b}$ and integral ideals $\mfrak{c}$.

One of the key differences in the theory of half-integral weight modular forms is in the congruence subgroups one considers. The factor of automorphy involved can only be defined for a certain subgroup $\mfrak{M}\leq M_{\mathbb{A}}$, and any congruence subgroups $\Gamma$ must therefore be contained in $\mfrak{M}$. This subgroup, $\mfrak{M}$, is defined via the theta series and is given by
\begin{align*}
	C_p^{\theta} :&= \{\xi\in D_p[1, 1]\mid (a_{\xi}b_{\xi}^T)_{ii}\in 2\mathbb{Z}_p, (c_{\xi}d_{\xi}^T)_{ii}\in 2\mathbb{Z}_p, 1\leq i\leq n\}, \\
	C^{\theta} :&= Sp_n(\mathbb{R})\prod_p C_p^{\theta}, \\
	\mfrak{M} :&= \{\sigma\in M_{\mathbb{A}}\mid \pr(\sigma)\in P_{\mathbb{A}}C^{\theta}\}.
\end{align*}
Typically we shall take $\mfrak{b}$ and $\mfrak{c}$ such that $\Gamma[\mfrak{b}^{-1}, \mfrak{bc}]\leq \Gamma[2, 2]\leq \mfrak{M}$.

The spaces defined above interact with each other as follows. The action of $Sp_n(\mathbb{R})$ on $\mathbb{H}_n$ and the traditional factor of automorphy are given by
\begin{align*}
	\gamma\cdot z = \gamma z :&= (a_{\gamma}z+b_{\gamma})(c_{\gamma}z+d_{\gamma})^{-1}, \\
	j(\gamma, z) :&= \det(c_{\gamma}z+d_{\gamma}),
\end{align*}
where $\gamma\in Sp_n(\mathbb{R})$ and $z\in\mathbb{H}_n$. If $\alpha\in G_{\mathbb{A}}$ then we extend the above by $\alpha\cdot z = \alpha_{\infty}\cdot z$ and $j(\alpha, z) = j(\alpha_{\infty}, z)$. If $\sigma\in M_{\mathbb{A}}$ with $\alpha = \pr(\sigma)\in G_{\mathbb{A}}$ then put $x_{\sigma} = x_{\alpha}$ for any $x\in~\{a, b, c, d\}$; the action of $M_{\mathbb{A}}$ on $\mathbb{H}_n$ is given by $\sigma\cdot z = \alpha\cdot z$.

For any $\sigma\in\mfrak{M}$ we can define a holomorphic function $h_{\sigma} = h(\sigma, \cdot):\mathbb{H}_n\to\mathbb{C}$ satisfying the following properties
\begin{align}
	h(\sigma, z)^2 &= \zeta j(\pr(\sigma), z)\ \text{for a constant $\zeta = \zeta(\sigma)\in \mathbb{T}$}, \label{h1} \\
	h(r_P(\gamma), z) &= |\det(d_{\gamma})_{\infty}^{\frac{1}{2}}|\ \text{if $\gamma\in P_{\mathbb{A}}$}, \label{h2} \\
	h(\rho\sigma\tau, z) &= h(\rho, z)h(\sigma, \tau z)h(\tau, z)\ \text{if $\pr(\rho)\in P_{\mathbb{A}}$ and $\pr(\tau)\in C^{\theta}$}. \label{h3}
\end{align}
The proofs for the above three properties can be found in \cite[pp. 294--295]{Shimuraold}. 

If $k$ is a half-integral weight then put $[k] := k-\frac{1}{2}\in\mathbb{Z}$; if $\ell$ is an integral weight then put $[\ell] := \ell$. The factors of automorphy of half-integral weights $k$ and integral weights $\ell$ are given as
\begin{align*}
	j_{\sigma}^k(z) :&= h_{\sigma}(z)j(\pr(\sigma), z)^{[k]}, \\
	j_{\alpha}^{\ell}(z) :&= j(\alpha, z)^{\ell},
\end{align*}
where $\sigma\in \mfrak{M}$, $\alpha\in G_{\mathbb{A}}$, and $z\in\mathbb{H}_n$. Given a function $f:\mathbb{H}_n\to\mathbb{C}$ and an element $\xi\in G_{\mathbb{A}}$ or $\mfrak{M}$, the \emph{slash operator} of an integral or half-integral weight $\kappa\in\frac{1}{2}\mathbb{Z}$ is defined by
\begin{align*}
	(f||_{\kappa}\xi)(z) :&= j_{\xi}^{\kappa}(z)^{-1}f(\xi\cdot z).
\end{align*}

\begin{definition}\label{modform} Let $\kappa\in\frac{1}{2}\mathbb{Z}$ be an integral or half-integral weight, and let $\Gamma\leq G$ be a congruence subgroup with the assumption that $\Gamma\leq\mfrak{M}$ if $\kappa\notin\mathbb{Z}$. Denote by $C_{\kappa}^{\infty}(\Gamma)$ the complex vector space of $C^{\infty}$ functions $f:\mathbb{H}_n\to\mathbb{C}$ such that $f||_{\kappa}\alpha = f$ for any $\alpha\in\Gamma$. Let $\mathcal{M}_{\kappa}(\Gamma)\subseteq C_{\kappa}^{\infty}(\Gamma)$ denote the subspace of holomorphic functions (with the additional cusp holomorphy condition if $n=1$). Elements of $\mathcal{M}_{\kappa}(\Gamma)$ are called \emph{modular forms of weight $\kappa$, level $\Gamma$}; if $\kappa\notin\mathbb{Z}$ they are also known as \emph{metaplectic modular forms}.
\end{definition}

Elements of $C_{\kappa}^{\infty}(\Gamma)$ and $\mathcal{M}_{\kappa}(\Gamma)$ have Fourier expansions summing over positive semi-definite symmetric matrices, the precise forms for which are given later in this section. The subspace $\mathcal{S}_{\kappa}(\Gamma)\subseteq\mathcal{M}_{\kappa}(\Gamma)$ is characterised by all forms $f$ such that the Fourier expansion of $f||_{\kappa}\sigma$ sums over positive definite symmetric matrices, for any $\sigma\in G_{\mathbb{A}}$ if $\kappa\in \mathbb{Z}$, or for any $\sigma\in\mfrak{M}$ if $\kappa\notin\mathbb{Z}$. Write
\[
	\mathcal{X}_{\kappa} = \bigcup_{\Gamma}\mathcal{X}_{\kappa}(\Gamma)
\]
where $\mathcal{X}\in\{\mathcal{M}, \mathcal{S}\}$ and the union is taken over all congruence subgroups of $G$ (that are contained in $\mfrak{M}$ if $\kappa\notin\mathbb{Z}$).

Take a fractional ideal $\mfrak{b}$ and integral ideal $\mfrak{c}$ and put $\Gamma = \Gamma[\mfrak{b}^{-1}, \mfrak{bc}]$; when $\kappa\notin\mathbb{Z}$ always make the crucial assumptions that 
\begin{align}
	\mfrak{b}^{-1}&\subseteq 2\mathbb{Z}, \label{BC1} \\
	\mfrak{bc}&\subseteq 2\mathbb{Z}, \label{BC2}
\end{align}
then we have $\Gamma \leq\mfrak{M}$ in this case.

By a Hecke character of $\mathbb{Q}$ we mean a continuous homomorphism $\psi:\mathbb{I}_{\mathbb{Q}}/\mathbb{Q}^{\times}\to\mathbb{T}$. Denote the restrictions to $\mathbb{R}^{\times}$, $\mathbb{Q}_p^{\times}$, and $\mathbb{Q}_{\mathbf{f}}^{\times}$ by $\psi_{\infty}$, $\psi_p$, and $\psi_{\mathbf{f}}$ respectively. We have that $\psi_{\infty}(x) = \sgn(x_{\infty})^t|x_{\infty}|^{i\nu}$ for $t\in\mathbb{Z}$ and $\nu\in\mathbb{R}$ and we say that $\psi$ is \emph{normalised} if $\nu = 0$. For any integral ideal $\mfrak{a}$ let $\psi_{\mfrak{a}} = \prod_{p\mid\mfrak{a}}\psi_p$.

Now take a normalised Hecke character $\psi$ of $\mathbb{Q}$ such that 
\begin{align}
	\psi_p(a) &= 1\ \text{if $a\in\mathbb{Z}_p^{\times}$ and $a\in 1+\mfrak{c}_p$}, \label{PC1} \\
	\psi_{\infty}(x)^n &= \sgn(x_{\infty})^{n[\kappa]}. \label{PC2}
\end{align}
Modular forms of character $\psi$ are then defined by
\begin{align*}
	C_{\kappa}^{\infty}(\Gamma, \psi) :&= \{F:\mathbb{H}_n\to\mathbb{C}\in C_{\kappa}^{\infty}\mid F||_{\kappa}\gamma = \psi_{\mfrak{c}}(|a_{\gamma}|)F\ \text{for all $\gamma\in\Gamma$}\}, \\
	\mathcal{M}_{\kappa}(\Gamma, \psi) :&= \mathcal{M}_{\kappa}\cap C_{\kappa}^{\infty}(\Gamma, \psi), \\
	\mathcal{S}_{\kappa}(\Gamma, \psi) :&= \mathcal{S}_{\kappa}\cap\mathcal{M}_{\kappa}(\Gamma, \psi).
\end{align*}

Understand $\pr = \id$ if $\kappa\in\mathbb{Z}$. If $f\in\mathcal{M}_{\kappa}(\Gamma, \psi)$ then its adelisation $f_{\mathbb{A}}:\pr^{-1}(G_{\mathbb{A}}) \to\mathbb{C}$ is
\[
	f_{\mathbb{A}}(x) := \psi_{\mfrak{c}}(|d_w|)(f||_{\kappa}w)(\mathbf{i})
\]
where $x = \alpha w$ for $\alpha\in G$ and $w\in \pr^{-1}(D[\mfrak{b}^{-1}, \mfrak{bc}])$, and $\mathbf{i} = iI_n$. To give the precise Fourier expansions of these forms, define the following spaces of symmetric matrices:
\begin{align*}
	S:&=\{\xi\in M_n(\mathbb{Q})\mid \xi^T=\xi\},
&&\hspace{10pt}S_+:=\{\xi\in S\mid \xi\geq 0\}, \\
	S^{\triangledown} :&= \{\xi\in S\mid \xi_{ii}\in\mathbb{Z}, \xi_{ij}\in\tfrac{1}{2}\mathbb{Z}, i<j\}, && \hspace{10pt}S_+^{\triangledown} := S^{\triangledown}\cap S_+, \\
	S(\mfrak{r}):&=S\cap M_n(\mfrak{r}) ,
&&S_{\mathbf{f}}(\mfrak{r}):=\prod_{p\in\mathbf{f}}S(\mfrak{r})_p,
\end{align*}
for any fractional ideal $\mfrak{r}$ of $\mathbb{Q}$.

Take a congruence subgroup $\Gamma = \Gamma[\mfrak{b}^{-1}, \mfrak{bc}]$ (contained in $\mfrak{M}$ if $\kappa\notin\mathbb{Z}$), a modular form $f\in\mathcal{M}_{\kappa}(\Gamma, \psi)$, and matrices $q\in GL_n(\AAQ)$, $s\in S_{\mathbb{A}}$. The Fourier expansion of $f_{\mathbb{A}}$ is given as
\[
	f_{\mathbb{A}}\left(r_P\begin{pmatrix} q & s\tilde{q} \\ 0 & \tilde{q}\end{pmatrix}\right) = |q_{\infty}|^{[k]}\|q_{\infty}\|^{\kappa-[\kappa]}\sum_{\tau\in S_+}c_f(\tau, q)e_{\infty}(\tr(\mathbf{i}q^T\tau q))e_{\mathbb{A}}(\tr(\tau s)),
\]
for some $c_f(\tau, q) = c(\tau, q; f)\in\mathbb{C}$ satisfying the following properties
\begin{align}
	&c_f(\tau, q)\neq 0\ \text{only if $e_{\mathbb{A}}(\tr(q^T\tau qs))=1$ for all $s\in S_{\mathbf{f}}(\mfrak{b}^{-1})$}, \label{c1} \\
	&c_f(b^T\tau b, q) = |b|^{[\kappa]}\|b\|^{\kappa-[\kappa]}c_f(\tau, bq)\ \text{for any $b\in GL_n(\mathbb{Q})$}, \label{c3} \\
	&\psi_{\mathbf{f}}(|a|)c_f(\tau, qa) = c_f(\tau, q)\ \text{for any $\diag[a, \tilde{a}]\in D[\mfrak{b}^{-1}, \mfrak{bc}]$}.  \label{c4}\hspace{60pt}
\end{align}

The proof of the above expansion and properties can be found in Proposition 1.1 of \cite{Shimurahalf}. The coefficients $c_f(\tau, 1)$ are the traditional Fourier coefficients of $f$ in the following sense. by property (\ref{c1}), the modular form $f\in\mathcal{M}_{\kappa}(\Gamma, \psi)$ has Fourier expansion
\[
	f(z) = \sum_{\tau\in S_+}c_f(\tau, 1)e(\tr(\tau z)),
\]
where $c_f(\tau, 1)\neq 0$ only if $\tau\in N(\mfrak{b})S_+^{\triangledown}$. If $F\in C_{\kappa}^{\infty}(\Gamma, \psi)$, then it has Fourier expansion of the form
\[
	F(z) = \sum_{\tau\in S} c_F(\tau, y) e(\tr(\tau x)),
\]
where $z = x+iy$ and the coefficients $c_F(\tau, y)$ are smooth functions of $y$ having values in $\mathbb{C}$. We have $c_F(\tau, y)$ is identically zero unless $\tau\in N(\mfrak{b})S^{\triangledown}$.

We finish this section with some final key definitions. Consider $\mfrak{b}$ fixed in the definitions of $\Gamma = \Gamma[\mfrak{b}^{-1}, \mfrak{bc}]$, so that this group depends only on $\mfrak{c}$, and let $\psi$ be a normalised Hecke character satisfying $(\ref{PC1})$ and $(\ref{PC2})$. For any two $f, g\in C_{\kappa}^{\infty}(\Gamma, \psi)$, the Petersson inner product is defined by
\[
	\langle f, g\rangle_{\mfrak{c}} = \langle f, g\rangle := \Vol(\Gamma\back\mathbb{H}_n)^{-1}\int_{\Gamma\back\mathbb{H}_n}f(z)\overline{g(z)}\Delta(z)^{\kappa}d^{\times}z
\]
in which
\[
	\Delta(z) := \det(\Im(z)), \hspace{15pt} d^{\times}z := \Delta(z)^{-n-1}\bigwedge_{i\leq j}(dx_{ij}\wedge dy_{ij}),
\]
for any $z = (x_{ij}+iy_{ij})_{i, j=1}^n\in\mathbb{H}_n$. This integral is convergent whenever one of $f, g$ belongs to $\mathcal{S}_{\kappa}(\Gamma, \psi)$.

\section{Complex $L$-function}
\label{Lfunsection}

\noindent The standard complex $L$-function associated to eigenforms is defined in this section, and the known Rankin-Selberg integral expression is stated. As in the previous section, take ideals $\mfrak{b}$ and $\mfrak{c}$ satisfying (\ref{BC1}) and (\ref{BC2}), and set $\Gamma = \Gamma[\mfrak{b}^{-1}, \mfrak{bc}]$.

For any Hecke character $\chi$ of $\mathbb{Q}$ let $\chi^*(p) = \chi^*(p\mathbb{Z})$ denote the associated ideal character. Though the integral expression can be stated for any half-integral weight $k$, we take $k\geq n+1$ to ease up on notation -- we shall be making this assumption later on anyway. For a prime $p$, the association of the Satake $p$-parameters -- an $n$-tuple $(\lambda_{p, 1}, \dots, \lambda_{p, n})\in\mathbb{C}^n$ -- to a non-zero Hecke eigenform $f\in\mathcal{S}_k(\Gamma, \psi)$ is well-known (see, for example, \cite[p. 46]{Shimurahalf}). Now set a Hecke character $\chi$ of conductor $\mfrak{f}$. The standard $L$-function of $f$, twisted by $\chi$, is then defined by
\begin{align*}
	L_p(t) :&= \begin{cases} \displaystyle\prod_{i=1}^n (1-p^n\lambda_{p, i}t) &\text{if $p\mid\mfrak{c}$}, \\ 
		\displaystyle\prod_{i=1}^n (1-p^n\lambda_{p, i}t)(1-p^n\lambda_{p, i}^{-1}t) &\text{if $p\nmid\mfrak{c}$}; \end{cases} \\
	L_{\psi}(s, f, \chi) :&= \prod_p L_p\left((\psi^{\mfrak{c}}\chi^*)(p)p^{-s}\right)^{-1},
\end{align*}
in which
\[
	\psi^{\mfrak{c}}(x) := \left(\frac{\psi}{\psi_{\mfrak{c}}}\right)(x).
\]
The Rankin-Selberg integral expression, (4.1) in \cite[p. 342]{Shimuraexp}, is given there in generality; we restate it now for our purposes. Fix $\tau\in N(\mfrak{b})S_+^{\triangledown}$ such that $c_f(\tau, 1)\neq 0$ and let $\rho_{\tau}$ be the quadratic character associated to the extension $\mathbb{Q}(i^{[n/2]}\sqrt{|2\tau|})$; choose $\mu\in\{0, 1\}$ such that $(\psi\chi)_{\infty}(x) = \sgn(x_{\infty})^{[k]+\mu}$. 

The key ingredients of the integral are three modular forms: the eigenform $f$, a theta series $\theta_{\chi}$, and a normalised non-holomorphic Eisenstein series $\mathcal{E}(z, s)$. To define the theta series, take any $0<\tau\in S$ and define an integral ideal $\mfrak{t}$ by the relation $h^T(2\tau)^{-1}h\in 4\mfrak{t}^{-1}$ for all $h\in \mathbb{Z}^n$. Take $\mu\in\{0, 1\}$ and a Hecke character $\chi$ such that $\chi_{\infty}(x)^n = \sgn(x_{\infty})^{n\mu}$. The theta series is then the sum
\begin{align}
	\theta_{\chi}^{(\mu)}(z; \tau) = \theta_{\chi}(z) := \sum_{x\in M_n(\mathbb{Z})}(\chi_{\infty}\chi^*)^{-1}(|x|)|x|^{\mu}e(\tr(x^T\tau xz)), \label{theta}
\end{align}
where we understand $(\chi_{\infty}\chi^*)(0) = 1$ if $f = \mathbb{Z}$ and as zero otherwise.
This has weight $\frac{n}{2}+\mu$, level $\Gamma[2, 2\mfrak{tf}^2]$ determined by Proposition 2.1 of \cite{Shimuraexp}, character $\rho_{\tau}\chi^{-1}$, and coefficients in $\mathbb{Q}(\chi)$.

The Eisenstein series of weight $\kappa\in\frac{1}{2}\mathbb{Z}$ is now defined in a little more generality. Let $\Gamma' = \Gamma[\mfrak{x}^{-1}, \mfrak{xy}]$ be a congruence subgroup, contained in $\mfrak{M}$ if $\kappa\notin\mathbb{Z}$, and let $\ph$ be a Hecke character satisfying (\ref{PC1}) with $\mfrak{y}$ in place of $\mfrak{c}$, and also such that $\ph_{\infty}(x) = \sgn(x_{\infty})^{[\kappa]}$ (note that this is a more stringent condition than the usual (\ref{PC2})). The Eisenstein series is defined by
\begin{align*}
	E(z, s; \kappa, \ph, \Gamma') := \sum_{\alpha\in P\cap\Gamma'\back\Gamma'}\ph_{\mfrak{y}}(|a_{\gamma}|)(\Delta^{s-\frac{\kappa}{2}}||_{\kappa}\alpha)(z), 
\end{align*}
where recall $\Delta(z) = \det(\Im(z))$, and we have $z\in\mathbb{H}_n$, $s\in\mathbb{C}$. This sum is convergent for $\Re(s)>\frac{n+1}{2}$ and can be continued meromorphically to all of $s\in\mathbb{C}$ by a functional equation with respect to $s\mapsto \frac{n+1}{2}-s$. This series belongs to $C_{\kappa}^{\infty}(\Gamma, \ph^{-1})$ and is normalised by a product of Dirichlet $L$-functions as follows. Let $\mfrak{a}$ be any integral ideal and define
\begin{align*}
	\begin{split}
		L_{\mfrak{a}}(s, \ph) :&= \prod_{p\nmid\mfrak{a}}(1-\ph^*(p)p^{-s})^{-1}; \\
		\Lambda_{\mfrak{a}}^{n, \kappa}(s, \ph) :&= \begin{cases}
			L_{\mfrak{a}}(2s, \ph)\displaystyle\prod_{i=1}^{[n/2]}L_{\mfrak{a}}(4s-2i, \ph^2) &\text{if $\kappa\in\mathbb{Z}$}, \\
			\displaystyle\prod_{i=1}^{[(n+1)/2]}L_{\mfrak{a}}(4s-2i+1, \ph^2) &\text{if $\kappa\notin\mathbb{Z}$}.
			\end{cases}
	\end{split}
\end{align*}
The normalised Eisenstein series is given by
\begin{align*}
	\mathcal{E}(z, s; \kappa, \ph, \Gamma') := \overline{\Lambda_{\mfrak{y}}^{n, \kappa}(s, \bar{\ph})}E(z, \bar{s}; \kappa, \ph, \Gamma').
\end{align*}

Set $\eta := \psi\bar{\chi}\rho_{\tau}$. In this setting, the integral expression of \cite[(4.1)]{Shimuraexp} becomes:
\begin{align}
	\begin{split}
		L_{\psi}(s, f, \bar{\chi}) &= \left[\Gamma_n\left(\tfrac{s-n-1+k+\mu}{2}\right)2c_f(\tau, 1)\right]^{-1}N(\mfrak{b})^{\frac{n(n+1)}{2}}|4\pi\tau|^{\frac{s-n-1+k+\mu}{2}} \\
	&\times \left(\tfrac{\Lambda_{\mfrak{c}}}{\Lambda_{\mfrak{y}}}\right)\left(\tfrac{2s-n}{4}\right)\prod_{q\in\mathbf{b}}g_q\left((\psi^{\mfrak{c}}\bar{\chi}^*)(q)q^{-s}\right)\langle f, \theta_{\bar{\chi}}\mathcal{E}(\cdot, \tfrac{2s-n}{4})\rangle_{\mfrak{y}} V,
	\end{split}\label{RS}
\end{align}
in which $\Lambda_{\mfrak{a}}(s) := \Lambda_{\mfrak{a}}^{n, k-\frac{n}{2}-\mu}(s, \eta)$; $\mfrak{y}:= \mfrak{c}\cap(2\mfrak{tf}^2)$; $\mathbf{b}$ is the finite set of primes $q\nmid\mfrak{c}$ such that $\ord_q(|\tau|)\neq 0$ and $g_q\in\mathbb{Z}[t]$ are certain polynomials satisfying $g_q(0) = 1$; 
\[
	\mathcal{E}(z, s) := \mathcal{E}\left(z, s; k-\tfrac{n}{2}-\mu, \bar{\eta}, \Gamma[\mfrak{b}^{-1}, \mfrak{by}]\right);
\]
and $V := \Vol(\Gamma[\mfrak{b}^{-1}, \mfrak{by}]\back\mathbb{H}_n)$.

\section{$p$-stabilisation}
\label{pstabsection}

\noindent Fixing a prime $p$, the initial key ingredient in our construction of the $p$-adic $L$-function is the replacement of an eigenform $f$ with its so-called $p$-stabilisation $f_0$. The form $f_0$ is also an eigenform away from $p$, whose eigenvalues there coincide with $f$, however it has the key property that $p$ divides the level of $f_0$ and is an eigenform for the operator $U_p$ -- the Atkin-Lehner operator that shifts Fourier coefficients. Thus the $L$-functions of $f$ and $f_0$ are easily relatable and so for full generality we can begin with an eigenform $f$, assume that $p\nmid\mfrak{c}$ does not divide the level, and then pass to $f_0$. In \cite{Mercuri} we constructed $f_0$ explicitly in the case $n=1$ which was possible through explicit formulae on the action of the Hecke operators involved on the Fourier coefficients. For general $n$ we modify the method of \cite{Panchishkin}, which involves abstract Hecke rings, the Satake isomorphism, and certain Hecke polynomials; at the end of this section however, we show how all this abstract Hecke yoga reduces to the explicit form found in \cite{Mercuri}, when $n=1$.

Let $k$ be a half-integral weight, $(\mfrak{b}^{-1}, \mfrak{bc})\subseteq 2\mathbb{Z}\times 2\mathbb{Z}$ be ideals, and $\psi$ be a Hecke character satisfying (\ref{PC1}) and (\ref{PC2}); put $\Gamma = \Gamma[\mfrak{b}^{-1}, \mfrak{bc}]$ and $D = D[\mfrak{b}^{-1}, \mfrak{bc}]$. Then define
\begin{align*}
	\So_p :&= GL_n(\mathbb{Z}_p),   &&\So := \prod_p GL_n(\mathbb{Z}_p), \\
	\Sx_p :&= M_n(\mathbb{Z}_p)\cap GL_n(\mathbb{Q}_p), &&\Sx := GL_n(\mathbb{Q})_{\mathbf{f}}\cap\prod_p \Sx_p, \\
	Z_0 :&= \{\diag[\tilde{q}, q]\mid q\in \Sx\}, &&Z:= D[2, 2]Z_0D[2, 2].
\end{align*}

If $(\Delta, \Xi)$ is a Hecke pair, in the sense of \cite[pp. 77 -- 78]{Andrianov}, then the abstract Hecke ring $\mathcal{R}(\Delta, \Xi)$ denotes the ring of formal finite sums $\sum_{\xi}c_{\xi}\Delta\xi\Delta$ where $c_{\xi}\in\mathbb{C}$ and $\xi\in \Xi$. Each double coset has a finite decomposition into single right cosets, and the law of multiplication is given in \cite[pp. 78 -- 79]{Andrianov}. Consider the Hecke ring $\mathcal{R}(V, W)$ defined in \cite[p. 39]{Shimurahalf} and let $\mathcal{R}$ denote the factor ring of $\mathcal{R}(V, W)$ defined in \cite[p. 41]{Shimurahalf} or analogously to (\ref{factorring}) below -- this is the adelic Hecke ring which acts on forms in $\mathcal{M}_k(\Gamma, \psi)$, and it is factored in order to give the Satake isomorphism. We need the use of a slightly different Hecke ring and we define this more explicitly. Let $D_0 := D\cap P_{\mathbb{A}}$ and $\Gamma_0 := \Gamma\cap P$; define
\begin{align*}
	Y_0 :&= \left\{\diag[\tilde{r}, r]\bigg| r\in\prod_p GL_n(\mathbb{Q}_p)\right\}, 
\end{align*}
and 
\begin{align*}
	W_0 :&= \{(\alpha, t)\mid t\in \mathbb{T}, \pr(\alpha)\in G_{\mathbf{f}}\cap D_0Y_0D_0\} \\
	V_0 :&= \{(\alpha, 1)\mid \pr(\alpha)\in G_{\mathbf{f}}\cap D_0\}.
\end{align*}
Now define the Hecke ring $\mathcal{S} := \mathcal{R}(V_0, W_0)$, which differs from $\mathcal{R}(V, W)$ of \cite{Shimurahalf} in allowing denominators of $p$ into the matrices $r$ defining $Y_0$ (contrast with the definition of $Z_0$), and is therefore analogous to the Hecke ring $L_0$ of \cite[pp. 81 -- 82]{Andrianov}. By Lemma 1.1.3 of \cite{Andrianov} there exists a $\mathbb{C}$-linear embedding $\eps : \mathcal{R}(V, W)\to\mathcal{S}$ defined on single cosets as $\eps(Dg) = D_0g$. The law of multiplication in $W_0$, and subsequent actions of $\mathcal{S}$ on $f_{\mathbb{A}}$ and $f\in\mathcal{M}_k(\Gamma, \psi)$, are defined in the same way as \cite[pp. 39 -- 41]{Shimurahalf}, and thus the factor ring
\begin{align}
	\mathcal{S}_0 := \mathcal{S}/\langle V_0(\alpha, 1)V_0-tV_0(\alpha, t)V_0\mid (\alpha, t)\in W_0\rangle \label{factorring}
\end{align}
also has a well-defined action on $f\in\mathcal{M}_k(\Gamma, \psi)$ and $f_{\mathbb{A}}$. The action of the double coset $D\diag[\tilde{q}, q]D\in\mathcal{R}(V, W)$ on $\mathcal{M}_k(\Gamma, \psi)$, for example, is given by first decomposing into single cosets
\[
	G\cap(D\diag[\tilde{q}, q]D) = \bigsqcup_{\alpha}\Gamma\alpha,
\]
where $\alpha\in G\cap D\diag[\tilde{q}, q]D$ and then summing over the actions of $\alpha$ on $f$ by the slash operator involving an extended factor of automorphy $J^k(\alpha, z)$ -- see Sections 2, 3, and 4 of \cite{Shimurahalf} for the details here.

	Let $A_q\in\mathcal{R}$ denote the image under projection of $V(\diag[\tilde{q}, q], 1)V\in\mathcal{R}(V, W)$, for $q\in\Sx$, and let $A_r\in\mathcal{S}_0$ denote the image under projection of $V_0(\diag[\tilde{r}, r], 1)V_0\in\mathcal{S}$, for $r\in\prod_p GL_n(\mathbb{Q}_p\cap\mathbb{Q})$. Then the local rings $\mathcal{R}_p$ and $\mathcal{S}_{0p}$ are the spaces generated by $A_q$ and $A_r$ respectively, where now $q\in\Sx_p$ and $r\in GL_n(\mathbb{Q}_p\cap\mathbb{Q})$. 

Assume $p\nmid\mfrak{c}$. Let  $W_n$ be the Weyl group of transformations generated by the transformations $x_i\mapsto x_i^{-1}; x_j\mapsto x_j$ for $j\neq i$, and let $\mathbb{C}[x_1^{\pm}, \dots, x_n^{\pm}]^{W_n}$ denote the ring of Weyl-invariant complex polynomials. The Satake map $\omega_p:\mathcal{R}_{p}\to\mathbb{C}[x_1^{\pm}, \dots, x_n^{\pm}]^{W_n}$ is defined in \cite[pp. 41 -- 42]{Shimurahalf} through the composition of two maps
\begin{align*}
	\omega_p:\mathcal{R}_{p}\xrightarrow{\Phi_p} \mathcal{R}(\So_p, GL_n(\mathbb{Q}_p))\xrightarrow{\omega_{0p}}\mathbb{C}[x_1^{\pm}, \dots, x_n^{\pm}]^{W_n},
\end{align*}
which we now give. By \cite[Lemma 1.2.2]{Andrianov} this is an isomorphism.

\paragraph{The map $\Phi_p$.}
If $\sigma = \diag[\tilde{q}, q]$ with $q\in\Sx_p$ then by Lemma 2.1 of \cite{Shimuraint} we have the decomposition
\begin{align}
	D_p\sigma D_{p} =\bigsqcup_{x\in X}\bigsqcup_{s\in Y_x}\bigsqcup_{d\in R_x}D_p\alpha_{d, s}, \hspace{20pt}\alpha_{d, s} = \begin{pmatrix} \tilde{d} & sd \\ 0 & d\end{pmatrix}, \label{satakecoset}
\end{align}
where $X\subseteq GL_n(\mathbb{Q}_p)$, $R_x\subset x\So_p$ represents $\So_p\back\So_p x\So_p$, and $Y_x\subseteq S_p$. To define $\Phi_p$ we extend, by $\mathbb{C}$-linearity, the map
\[
	\Phi_p(D_p\sigma D_p) := \sum_{d, s}J(r_P(\alpha_{d, s}))^{-1}\So_pd,
\]
where $J(\alpha) := J^{\frac{1}{2}}(\alpha, i)$ for $\alpha\in \pr^{-1}(D[2, 2]Z_0D[2, 2])$ (see \cite[(2.7) and Lemma 2.4]{Shimurahalf} for the precise definition and characterisation of $J(\alpha)$). By \cite[Lemma 4.3]{Shimurahalf} the map $\Phi_p$ is injective.

\paragraph{The map $\omega_{0p}$.} Note that any coset $\So_pd$ with $d\in GL_n(\mathbb{Q}_p)$ contains an upper triangular matrix of the form
\begin{align*}
	\begin{pmatrix} 
	p^{a_{d_1}} & \star & \cdots & \star \\ 
	0 & p^{a_{d_2}} & \cdots & \star \\ \vdots & \vdots & \ddots & \vdots \\ 
	0 & 0 & \cdots & p^{a_{d_n}}
	\end{pmatrix},  
\end{align*}
with $a_{d_i}\in\mathbb{Z}$, and then define
\[
	\omega_{0p}(\So_pd) := \prod_{i=1}^n (p^{-i}x_i)^{a_{d_i}}.
\]
Through the decomposition $\So_p x\So_p = \bigsqcup_d \So_pd$ and $\mathbb{C}$-linearity, we extend this to obtain $\omega_{0p}$.
\par\bigskip
By multiplying out elements of $\diag[\tilde{r}, r](D_0)_p$, for $r\in GL_n(\mathbb{Q}_p\cap\mathbb{Q})$, we see that $(D_0)_p\diag[\tilde{r}, r](D_0)_p$ also has a single coset decomposition of the form (\ref{satakecoset}). Thus we can analogously define $\Phi_p':\mathcal{S}_{0p}\to GL_n(\So_p, GL_n(\mathbb{Q}_p))$ and 
\[
	\omega_p':=\omega_{0p}\circ\Phi_p':\mathcal{S}_{0p}\to \mathbb{C}[x_1^{\pm}, \dots, x_n^{\pm}].
\]
The map $\Phi_p'$, and therefore $\omega_p'$, is no longer necessarily injective. There is a local embedding $\eps_{0p}:\mathcal{R}_p\to\mathcal{S}_{0p}$ and we have $\omega_p = \omega_p'\circ\eps_{0p}$. There exists $U_p\in\mathcal{S}_{0p}$ -- called the \emph{Frobenius} element -- defined by
\[
	U_p = \Gamma_0\begin{pmatrix} p^{-1}I_n & 0 \\ 0 & pI_n\end{pmatrix} \Gamma_0 =  \bigsqcup_{u\in S(\mfrak{b}^{-1}/p^2\mfrak{b}^{-1})}\Gamma_0\begin{pmatrix} p^{-1}I_n & p^{-1}u \\ 0 & pI_n\end{pmatrix}.
\]
If $n=1$, it is well known that $U_p$ corresponds to the $p$th Hecke operator when $p\mid\mfrak{c}$; for general $n>1$, this is no longer true. Note that $\omega_p'(U_p) = p^{\frac{n(n+1)}{2}}x_1\cdots x_n$. 

Let $\mathcal{C} := \{A\in \mathcal{S}_{0p}\mid U_pA = AU_p\}$ denote the centraliser of $U_p$ in $\mathcal{S}_{0p}$. The map $\Phi_p'$ is injective when restricted to $\mathcal{C}$ by the following argument. 

\begin{proposition}\label{centraliserform} Any $A\in\mathcal{C}$ is a linear combination of double cosets
\[
	(D_0)_p\begin{pmatrix} \tilde{r} & 0 \\ 0 & r\end{pmatrix}(D_0)_p,
\]
where $r\in M_n(\mathbb{Z}_p)\cap GL_n(\mathbb{Q}_p)$. 
\end{proposition}

\begin{proof} This is essentially the second statement of Proposition 2.1.1 of \cite{Andrianov} with $\delta = 0$ (in the notation of Andrianov). To prove it, define $U_p^- := \Gamma_0\begin{psmallmatrix} pI_n & 0 \\ 0 & p^{-1}I_n\end{psmallmatrix}\Gamma_0$ then multiply out the cosets of both sides of the relation $U_p^-A_r = A_rU_p^-$ to see that $r$ must have entries in $\mathbb{Z}_p^{\times}[p^{-1}]$. Define the involution $A_r^{\iota} := A_{r^{-1}}$, which satisfies $U_p^{\iota} = U_p^-$ and apply it to the condition $U_pA_r = A_rU_p$ to obtain the proposition.
\end{proof}

\begin{proposition}\label{Phipinj} The map $\Phi_p'$ is injective when restricted to $\mathcal{C}$.
\end{proposition}

\begin{proof} By the previous proposition, if $A_r\in\mathcal{C}$ then $r\in M_n(\mathbb{Z}_p)\cap GL_n(\mathbb{Q}_p)$. We therefore have the decomposition
\[
	(D_0)_p\diag[\tilde{r}, r](D_0)_p = \bigsqcup_{d, s}(D_0)_p\alpha_{d, s},
\]
where $\alpha_{d, s}$ is as in (\ref{satakecoset}), ranging over $d\in \So_p\back \So_pr\So_p$ and $s\in S(\mfrak{b}^{-1})_p/d^TS(\mfrak{b}^{-1})_pd$. This is easily seen by multiplying out $\diag[\tilde{r}, r](D_0)_p$ for such $r$ and is analogous to the case $p\mid\mfrak{c}$ in \cite[Lemma 2.3]{Shimurahalf}. Now $J(r_P(\alpha_{d, s}) = 1$ by Lemma 2.4 of \cite{Shimurahalf}, and therefore $\Phi_p'(A_r) = |\det(r)|_{\mathbb{A}}^{-n-1}\So_pr\So_p$, which shows injectivity.
\end{proof}

The definition of the $p$-stabilisation is now achieved through factorisations of a certain Hecke polynomial. This polynomial is an element $\widetilde{R}_n\in\mathbb{C}[x_1^{\pm}, \dots, x_n^{\pm}][z]$ defined by:
\begin{align*}
	\widetilde{R}_n(x_1, \dots, x_n; z) = \widetilde{R}_n(z) :&= \prod_{\delta_i\in\{\pm 1\}}(1-p^{\frac{n(n+1)}{2}}x_1^{\delta_1}\cdots x_n^{\delta_n}z), \\
\end{align*}
It has an immediate decomposition of the form
\[
	\widetilde{R}(z) = \sum_{m=0}^{2^n}(-1)^m\widetilde{T}_mz^m,
\]
where $\widetilde{T}_m = \widetilde{T}_m(x_1, \dots, x_n)\in\mathbb{C}[x_1^{\pm}, \dots, x_n^{\pm}]$. By definition of $\widetilde{R}_n$, the coefficients $\widetilde{T}_m$ are clearly invariant under the group of Weyl transformations so, by the Satake isomorphism, there exists a polynomial
\begin{align}
	R_n(z) = \sum_{m=0}^{2^n}(-1)^mT_mz^m \label{RTfact}
\end{align}
whose coefficients $T_m\in\mathcal{R}_p$ satisfy $\widetilde{T}_m = \omega_p(T_m)$. If $n=1$ and $T_p$ denotes the $p$th Hecke operator then notice, from \cite[(5.4a)]{Shimurahalf}, that $\omega_p(T_1) = \omega_p(T_p)$, and so $T_1 = T_p$ in this case.

The polynomial $R_n(z)$ has a factorisation involving the Frobenius element $U_p$ -- this part is similar to the methods found in \cite[pp. 90--91]{Andrianov} and \cite[pp. 42--50]{Panchishkin}.

\begin{lemma} \label{usum} With $T_m$ and $U_p$ defined as above
\[
	\sum_{m=0}^{2^n}(-1)^mT_mU_p^{2^n-m} = 0.
\]
\end{lemma}

\begin{proof} Denote the sum on the left-hand side by $Y$, this belongs to $\mathcal{S}_{0p}$. It is easy to check that
$\widetilde{R}_n(z) = (p^{\frac{n(n+1)}{2}}z)^{2^n}\widetilde{R}_n((p^{n(n+1)}z)^{-1})$ so, immediately from (\ref{RTfact}) we have
\begin{align}
	T_m = p^{n(n+1)(m-2^{n-1})}T_{2^n-m}. \label{Trel}
\end{align}
By (\ref{Trel}) above and the same argument of Proposition 2.1.2 in \cite[pp. 88--89]{Andrianov} we therefore have $T_m U_p^{2^n-m}\in\mathcal{C}$. Since $\Phi_p'$ is injective on $\mathcal{C}$ by Proposition \ref{Phipinj}, we just need to show that $\omega_p'(Y) = 0$. For this, note
\begin{align*}
	\omega_p'(Y) &= \omega_p'(U_p)^{2^n}\sum_{m=0}^{2^n}(-1)^m\widetilde{T}_m\cdot (\omega_p'(U_p)^{-1})^m \\
	&= \omega_p'(U_p)^{2^n}\widetilde{R}_n(\omega_p'(U_p)^{-1})
\end{align*}
which is zero, since $(1-\omega_p'(U_p)z) = (1-p^{\frac{n(n+1)}{2}}x_1\cdots x_nz)$ is a factor of $\widetilde{R}_n(z)$.
\end{proof}

For any $0\leq m\leq 2^n$, define
\[
	V_{m, p} = V_m := \sum_{\ell=0}^m (-1)^{\ell}T_{\ell}U_p^{m-\ell}\in\mathcal{S}_{0p}.
\]

\begin{proposition}\label{Rfact} The Hecke polynomial $R(z)$ can be factorised as
\begin{align}
	R(z) = \left(\sum_{m=0}^{2^n-1}V_mz^m\right)(1-U_pz) \label{Rfacteq}
\end{align}
\end{proposition}

\begin{proof} By definition $V_0 = 1$ and by Lemma \ref{usum} $V_{2^n-1}U_p = -T_{2^n}$. For the rest, $1\leq m\leq 2^n-2$, we have
\[
	V_m-V_{m-1}U_p = \sum_{\ell=0}^m (-1)^{\ell}T_{\ell}U_p^{m-\ell}-\sum_{\ell=0}^{m-1}(-1)^{\ell}T_{\ell}U_p^{m-1-\ell}U_p = (-1)^mT_m.
\]
Expanding the right hand side of (\ref{Rfacteq}) therefore gives the factorisation (\ref{RTfact}), which concludes the proof.
\end{proof}

\begin{definition}\label{pstab} Let $f\in\mathcal{M}_k(\Gamma, \psi)$ be a non-zero Hecke eigenform with Satake $p$-parameters $(\lambda_{p, 1}, \dots, \lambda_{p, n})$, assuming $p\nmid\mfrak{c}$. Set
\[
	\lambda_0 := p^{\frac{n(n+1)}{2}}\lambda_{p, 1}\cdots \lambda_{p, n}.
\]
Then the $p$-stabilisation of $f$ is defined by
\begin{align}
	f_0 := \sum_{m=0}^{2^n-1}\lambda_0^{-m}f|V_{m, p}. \label{pstabform}
\end{align}
\end{definition}

\begin{proposition} \label{pstabprop} If $f\in\mathcal{M}_k(\Gamma[\mfrak{b}^{-1}, \mfrak{bc}], \psi)$ is an eigenform and $p\nmid\mfrak{c}$, then we have that $f_0\in\mathcal{M}_k(\Gamma[\mfrak{b}^{-1}, \mfrak{bc}_0], \psi)$, where $\mfrak{c}_0 = \mfrak{c}p^{2(2^n-1)}$. Moreover,
\[
	f_0|U_p = \lambda_0 f_0.
\]
\end{proposition}

\begin{proof} Recall $U_p^- := \Gamma_0\begin{psmallmatrix} pI_n & 0 \\ 0 & p^{-1}I_n\end{psmallmatrix}\Gamma_0\in\mathcal{S}_{0p}$; clearly $f|U_p^- = p^{nk}f(p^2z)$ has level $\Gamma[\mfrak{b}^{-1}, \mfrak{bc}p^2]$ and therefore, as operators, $U_p^-\Gamma[\mfrak{b}^{-1}, \mfrak{bc}p^2] = U_p^-$. Recall $\iota$ as the involution on $\mathcal{S}_{0p}$ defined by $A_r^{\iota} = A_{r^{-1}}$, which satisfies $U_p = (U_p^-)^{\iota}$. So, by the argument of \cite[p. 49]{Panchishkin}, we have $U_p\Gamma[\mfrak{b}^{-1}, \mfrak{bc}p^2] = U_p$ as operators as well. So we have that $f|U_p\in\mathcal{M}_k(\Gamma[\mfrak{b}^{-1}, \mfrak{bc}p^2], \psi)$ and the first property follows by definition of $f_0$ and $V_m$.

The action of $\alpha\in\mathbb{C}$ on $f$ is considered the scalar one, i.e. $f|\alpha = \alpha f$. The second property is then given by the calculation
\begin{align*}
	f_0|[\lambda_0-U_p] &= \lambda_0 f_0|[1-\lambda_0^{-1}U_p] \\
		&=\lambda_0 f\bigg|\left[\sum_{m=0}^{2^n-1}(\lambda_0^{-1})^mV_{m, p}\right][1-\lambda_0^{-1}U_p] \\
	&=\lambda_0 f|R(\lambda_0^{-1}),
\end{align*}
where Proposition \ref{Rfact} was invoked in the last line and Definition \ref{pstab} in the penultimate. This is zero since we have that $f|R(\lambda_0^{-1}) = \widetilde{R}(\lambda_{p, 1}, \dots, \lambda_{p, n}; \lambda_0^{-1})f$ and that $(1-\lambda_0z)$ is a factor of $\widetilde{R}(\lambda_{p, 1}, \dots, \lambda_{p, n}; z)$.
\end{proof}

For $q\neq p$, the $q$th Hecke operator commutes with $V_{m, p}$. Therefore $f_0$ and $f$ share the same eigenvalues away from $p$, and we then have the following corollary.

\begin{corollary}\label{pstablfun} Assume that $f_0\neq 0$. If $1\leq \ell\in\mathbb{Z}$ and $\chi$ is a character of conductor $p^{\ell}$, then $L_{\psi}(s, f, \chi) = L_{\psi}(s, f_0, \chi)$. 
\end{corollary}

In \cite{Mercuri} we showed, if $n=1$, that the $p$-stabilisation of $f$ takes the form
\begin{align}
	f_0(z) := f(z) - \left(\frac{-1}{p}\right)^{[k]}p^{-\frac{1}{2}}\lambda_{p, 1}^{-1}\left(f\otimes\left(\frac{\cdot}{p}\right)\right)(z) - p^{k-1}\lambda_{p, 1}^{-1}f(p^2z), \label{pstab1}
\end{align}
where, for any Dirichlet character $\ph$ of conductor $F$,
\[
	(f\otimes\ph)(z) := \sum_{n=1}^{\infty}\ph(n)c_f(n, 1)e(nz)
\]
denotes the twist of $f$ by $\ph$. This satisfies $f_0|U_p = p\lambda_{p, 1}f_0$ by direct construction. If $\mfrak{bc}\subseteq\mfrak{b}^{-1}$ (for example, if $\mfrak{b} = 2^{-1}\mathbb{Z}$), then $f\otimes\ph\in\mathcal{S}_k(\Gamma[\mfrak{b}^{-1}, F^2\mfrak{bc}], \psi\ph^2)$, so we can see immediately that $f_0\in\mathcal{S}_k(\Gamma[\mfrak{b}^{-1}, \mfrak{bc}p^2], \psi)$ and this matches the first part of Proposition \ref{pstabprop}.

By definition we have $V_{1, p} = U_p-T_1 = U_p-T_p$ in this case, so the abstract definieion of $f_0$ in Definition \ref{pstab}, when we set $n = 1$, becomes
\[
	f_0 = f+(p\lambda_{p, 1})^{-1}f|V_{1, p} = f+p^{-1}\lambda_{p, 1}^{-1}f|U_p-p^{-1}\lambda_{p, 1}^{-1}\Lambda(p)f
\]
where $\Lambda(p)$ denotes the eigenvalue of $f$ under $T_p$. By Lemma 3.1 (c) of \cite{Mercuri}, this is precisely the form of (\ref{pstab1}) above.

\paragraph{Non-vanishing of $f_0$.} It is not clear from the above method that $f_0\neq 0$ if $f\neq 0$. That $f_0$ may vanish is entirely possible, as is remarked in \cite[p. 50]{Panchishkin}.

Suppose that $\Lambda:\mathcal{R}\to\mathbb{C}$ is a homomorphism defining the eigenvalues of $f$, that is for all $1\leq m\leq 2^n$ we have $f|T_m = \Lambda(T_m)f$. By the definition in (\ref{pstabform}) and of $V_{m, p}$ we get
\begin{align*}
f_0 &= \sum_{m=0}^{2^n-1}\lambda_0^{-m}\sum_{\ell = 0}^m(-1)^{\ell}\Lambda(T_{\ell})f|U_p^{m-\ell} \\
	&= \sum_{v=0}^{2^n-1}\left[\sum_{u=v}^{2^n-1}(-1)^{u-v}\Lambda(T_{u-v})\lambda_0^{-u}f\right]|U_p^{v}.
\end{align*}
Assume that $f\neq 0$, so that we can take $\tau\in S_+$ such that $c_f(\tau, 1)\neq 0$. Using the fact that $c(\tau, 1; f|U_p) = p^{n(n+1-k)}c_f(p^2\tau, 1)$, the above formulation of $f_0$ gives
\begin{align}
	c_{f_0}(\tau, 1) = \sum_{v = 0}^{2^n-1}\left[\sum_{u=v}^{2^n-1}(-1)^{u-v}\Lambda(T_{u-v})\lambda_0^{-u}\right]p^{\frac{n(n+1)}{2}(2-k)v}c_f(p^{2v}\tau, 1). \label{coeffsum}
\end{align}
The above formula may be used as a method of checking, computationally, whether one has $c_{f_0}(\tau, 1)\neq 0$ as well. Given the formula in (\ref{coeffsum}) above, it seems unlikely that $c_{f_0}(\tau, 1)$ should vanish for all $\tau$ outside of a few special cases. As an example, consider the $n = 1$ case and assume that $c_f(\tau, 1)\neq 0$ for some $0<\tau\in\mathbb{Z}$ such that $p^2\nmid\tau$. By (\ref{pstab1}) the coefficient $c_{f_0}(\tau, 1) = 0$ only if $p\lambda_{p, 1} = \left(\frac{-1}{p}\right)^{[k]}\left(\frac{n}{p}\right)\sqrt{p}$. This becomes less trivial a situation if $c_f(\tau, 1)\neq 0$ only for $p^2\mid\tau$. As things become significantly more complex for general $n$, we acknowledge that this does not constitute a particularly strong argument, but it is hopefully enough to convince the reader that there should exist eigenforms $f\neq 0$ for which $f_0\neq 0$ as well.

In \cite[Section 9]{Bocherer}, B\"{o}cherer and Schmidt give an alternative construction for the $p$-stabilisation of a Siegel modular form of integral weight, which does guarantee that $f_0\neq 0$. Though this is perhaps stronger than our construction, one still needs to make an assumption that such a non-zero $f_0$ should exist and this is incorporated into B\"{o}cherer-Schmidt's definition of $p$-regular \cite[p. 1431]{Bocherer}. Their construction takes two Andrianov-type identities of Dirichlet series for $f$ and $f_0$ and uses them to compare their Satake parameters directly. It has a fairly simple generalisation to the present setting by using the identity of \cite[Corollary 5.2]{Shimurahalf}. Indeed this identity becomes almost exactly the same as that of \cite[Proposition 9.1]{Bocherer} by putting $[|x|\mathbb{Z}] = Y^{\ord_p(|x|)}$ and $[v] = Y$ in the notations found in \cite{Shimurahalf}, as well as in the definition of $D(\tau, p; f)$ in \cite[Theorem 5.1]{Shimurahalf}. All that remains is to manipulate the lattice sum, the far right-hand component of \cite[Corollary 5.2]{Shimurahalf}, and express it as a sum of the $U(\pi_{i})$ Hecke operators (defined as the double coset $\Gamma_0\diag[\tilde{\pi_i}, \pi_i]\Gamma_0$ and $\pi_i = \diag[pI_i, I_{n-i}]$). This was done for the Hermitian modular forms in \cite[Section 7]{BouganisHerm}, but remains the same for our case.

\section{Tracing the Rankin-Selberg Integral}
\label{tracesection}

Given the relationship, established in Corollary \ref{pstablfun}, between $L(s, f, \chi)$ and $(s, f_0, \chi)$ the focus can be shifted to the latter. The level, $\mfrak{y}$, of the Rankin-Selberg integral (\ref{RS}) will depend on $\chi$, which dependence we naturally seek to avoid. This is achieved in this section by making crucial use of the behaviour of $f_0$ under $U_p$.

Fix $0<\tau\in S_+$ such that $c_{f_0}(\tau, 1)\neq 0$. Recall $\mfrak{t}$ as an integral ideal such that $h^T(2\tau)^{-1}h\in 4\mfrak{t}^{-1}$ and define
\begin{align}\label{tauhat}
	\hat{\tau} := N(\mfrak{t})(2\tau)^{-1}\in M_n(\mathbb{Z}).
\end{align}
Take a Dirichlet character $\chi$ of modulus $p^{\ell}$ and conductor $p^{\ell_{\chi}}$ with $0\leq \ell_{\chi}\leq \ell\in\mathbb{Z}$, choose a $\mu\in\{0, 1\}$ such that $(\psi_{\infty}\chi)(-1) = (-1)^{[k]+\mu}$, and put $\eta := \psi\bar{\chi}\rho_{\tau}$. 

This section involves many levels and liftings of modular forms through these levels, so first we define and clarify these schematically. Fix $\mfrak{b}$ and note by (\ref{c1}) that $\mfrak{b}^{-1}\mid\mfrak{t}$, so we can think of $f_0$ as a form of level $\Gamma[\mfrak{b}^{-2}, \mfrak{b}^2\mfrak{t}\mfrak{c}_0]$ and put
\[
	\mfrak{y}_{\chi} := [\mfrak{tc}p^{\ell_{\chi}}]^2.
\]
The ideal $\mfrak{y}_{\chi}$ can be taken as the level of the integral in the Rankin-Selberg expression of $L_{\psi}(s, f_0, \chi)$ only if $\ell_{\chi}\geq 2^n-1$; to avoid this condition we generally choose higher levels. The levels involved are $\Gamma_{\alpha} := \Gamma[\mfrak{b}^{-2}, \mfrak{b}^2\mfrak{y}_{\alpha}]$ where the integral ideals $\mfrak{y}_{\alpha}$ are indexed by $\alpha\in\{r, \ell\in\mathbb{Z}\mid \ell_{\chi}\leq\ell\leq r\}\cup\{0\}$. They are defined below, arranged in order of divisibility:
\begin{equation*}
	\begin{split}
		&\mfrak{y}_{r} := \mfrak{y}_{0}p^{2r} \\
		&\rotatebox[origin=c]{270}{$\subseteq$} \\
		&\vdots \\
		&\rotatebox[origin=c]{270}{$\subseteq$} \\
		&\mfrak{y}_{\ell} := \mfrak{y}_{0}p^{2\ell} \\
		&\rotatebox[origin=c]{270}{$\subseteq$} \\
		&\vdots \\
		&\rotatebox[origin=c]{270}{$\subseteq$} \\
		&\mfrak{y}_{1} := \mfrak{y}_0p^{2} \\ 
		&\rotatebox[origin=c]{270}{$\subseteq$} \\
		&\mfrak{y}_0 := \mfrak{t}^2\mfrak{cc}_0.
	\end{split}
\end{equation*}
Later on, when we invoke the Kummer congruences, we shall take a set of Dirichlet characters of varying moduli $p^{\ell}$ and we shall be considering a sum of Rankin-Selberg integral expressions of varying levels $\mfrak{y}_{\ell}$. Then we shall take a single $r\geq 0$ so that all characters in the set are defined modulo $p^r$ and therefore we can simply lift all the Rankin-Selberg integrals of varying levels to all be of the same level $\mfrak{y}_r$ and finally we trace the Rankin-Selberg integral back down to $\mfrak{y}_0$ which process is given in the rest of this section. This is so that we can treat all characters uniformly. In specific cases, i.e. when we consider a single primitive Dirichlet character with $\ell = \ell_{\chi}\geq 2^n-1$, one need not lift up to $r$ in the first place and such a case is given as an example at the end of this section but will not be of much use later on. 

Assuming that $\chi$ is a Dirichlet character of modulus $p^{\ell}$ with $\ell\geq 1$, the Rankin-Selberg expression from \cite[(4.1)]{Shimuraexp} of $L_{\psi}(s, f_0, \bar{\chi})$ is given as
\begin{align}
	\begin{split}
	L_{\psi}(s, f_0, \bar{\chi}) &= \left[\Gamma_n\left(\tfrac{s-n-1+k+\mu}{2}\right)2c_{f_0}(\tau, 1)\right]^{-1}N(\mfrak{b})^{n(n+1)}|4\pi\tau|^{\frac{s-n-1+k+\mu}{2}} \\
	&\hspace{10pt}\times\left(\frac{\Lambda_{\mfrak{c}_0}}{\Lambda_{\mfrak{y}_0}}\right)\left(\tfrac{2s-n}{4}\right) \prod_{q\in\mathbf{b}}g_q\left((\psi^{\mfrak{c}_0}\bar{\chi})(q)q^{-s}\right)\left\langle f_0, \theta_{\bar{\chi}}\mathcal{E}(\cdot, \tfrac{2s-n}{4})\right\rangle_{\mfrak{y}_{r}}V_{r},
	\end{split} \label{f0RS}
\end{align}
in which $r\geq \ell$ and $V_{r} := \Vol(\Gamma[\mfrak{b}^{-2}, \mfrak{b}^2\mfrak{y}_{r}]\back\mathbb{H}_n)$.

Write $Y_{\alpha} := N(\mfrak{b})\sqrt{N(\mfrak{y}_{\alpha})}\in\mathbb{Z}$ for $\alpha\in\{0, \ell, r, \chi\}$; then $Y_0 = N(\mfrak{tbc})p^{2^n-1}$, and note that $Y_{\chi} =~Y_0p^{\ell_{\chi}}$, $Y_{\ell} = Y_0p^{\ell}$, and $Y_r = Y_0p^r$. Also $Y_{\chi} = Y_0p^{\ell_{\chi}-2^n-1}$ if $\ell_{\chi}\geq 2^n-1$.

The definition of the trace map on modular forms is well known; with $\mfrak{b}$ fixed, the map $\Tr_{\mfrak{c}_1}^{\mfrak{c}_2}$ for any $\mfrak{c}_2\subseteq\mfrak{c}_1$ takes modular forms in $\mathcal{M}_k(\Gamma_2, \psi)$ down to forms in $\mathcal{M}_k(\Gamma_1, \psi)$, where $\Gamma_i = \Gamma[\mfrak{b}^{-2}, \mfrak{b}^2\mfrak{c}_i]$, and is defined by decomposing $\Gamma_1 = \bigsqcup_{\gamma}\Gamma_2\gamma$ and summing over all the slash operator actions by these coset representatives. If $g\in~\mathcal{M}_{\frac{n}{2}+\mu}(\Gamma[\mfrak{b}^{-2}, \mfrak{b}^2\mfrak{y}_r], \chi\rho_{\tau})$, then put $F_g(z, s) := g(z)\mathcal{E}(z, \frac{2s-n}{4}; k-\frac{n}{2}-\mu, \bar{\eta}, \Gamma_r)$ and we have
\begin{align*}
	\Tr_{\mfrak{y}_0}^{\mfrak{y}_r}(F_g) = \sum_{u\in S(\mathbb{Z}/p^{2r}\mathbb{Z})} F_g\big\|_k\begin{psmallmatrix} I_n & 0 \\ N(\mfrak{y}_0)u & I_n\end{psmallmatrix} = \sum_{u\in S(\mfrak{b}^{-2}/p^{2r}\mfrak{b}^{-2})}F_g\big\|_k\begin{psmallmatrix} I_n & 0 \\ Y_0^2u & I_n\end{psmallmatrix}.
\end{align*}
Define, for any $M\in\mathbb{Z}$, the matrix
\[
	\iota_M := \begin{pmatrix} 0 & -M^{-1}I_n \\ MI_n & 0\end{pmatrix},
\]
which belongs to $P\iota$ and is therefore in $\mfrak{M}$. Associate to $\iota_M$ the operator $W(M)$, acting on any modular form $h$ of weight $\kappa\in\frac{1}{2}\mathbb{Z}$ by $h|W(M) = h\|_{\kappa}\iota_M$.

\begin{proposition}\label{tracereduce} Let $\chi$ be of modulus $p^{\ell}$, and let $g$ and $F_g$ be as above. If $r\geq 0$ is an integer, then
\[
	\left\langle f_0, F_g(\cdot, s)\right\rangle_{\mfrak{y}_r} = (-1)^{n[k]}\left\langle f_0, H_g(\cdot, s)|U_p^rW(Y_0)\right\rangle_{\mfrak{y}_0},
\]
where $H_g := F_g|W(Y_r)$.
\end{proposition}

\begin{proof} By the definition of the trace map and substitution of variables in the integral, we have
\[
	\left\langle f_0, F_g(\cdot, s)\right\rangle_{\mfrak{y}_r} = \left\langle f_0, \Tr_{\mfrak{y}_0}^{\mfrak{y}_r}(F_g)\right\rangle_{\mfrak{y}_0}.
\]
To finish, note that $W(Y_0)^2 = (-1)^{n[k]}$ and we claim $\Tr_{\mfrak{y}_0}^{\mfrak{y}_r}(F_g)|W(Y_0) = H_g|U_p^r$, the proof of which, in contrast to the integral-weight case, is twofold. That the matrices corresponding to the operators match is given by the simple matrix multiplication
\[
		\begin{pmatrix} I_n & 0 \\ Y_0^2u & I_n\end{pmatrix}\iota_{Y_0} = \iota_{Y_r}\begin{pmatrix} p^{-r}I_n & -p^{-r}u \\ 0 & p^{r}I_n\end{pmatrix},
\]
for $u\in S(\mfrak{b}^{-2}/p^{2r}\mfrak{b}^{-2})$ and in which we used $Y_r = Y_0p^r$. For the claim to hold however, we need to check that the half-integral weight factors of automorphy match up as well, for which the requisite identity is
\begin{align}
	h\left(\begin{psmallmatrix} I_n & 0 \\ Y_0^2u & I_n\end{psmallmatrix}, \iota_{Y_0}z\right)h(\iota_{Y_0}, z) = h(\iota_{Y_r}, \alpha_uz)J^{\frac{1}{2}}(\alpha_u, z), \label{factoridentity}
\end{align}
where $\alpha_u = \begin{psmallmatrix} p^{-r}I_n & -p^{-r}u \\ 0 & p^rI_n\end{psmallmatrix}$. We have $h(\iota_M, z) = |Miz|^{\frac{1}{2}}$ by considering $\iota_M\in P\iota$ and using (\ref{h2}), (\ref{h3}), and \cite[(2.5)]{Shimuratheta}. Per the definition of $J^k$ in \cite[(2.7)]{Shimurahalf} write $\alpha_u = \mfrak{z}\xi$, where $\mfrak{z}\in Z_0$ and $\xi\in D[2, 2]$ are defined by $\mfrak{z}_{\infty} := I_{2n}$, $\xi_{\infty} := \alpha_u$, $\mfrak{z}_q = \begin{psmallmatrix} p^{-r}I_n & 0 \\ 0 & p^rI_n\end{psmallmatrix}$, $\xi_q = \begin{psmallmatrix} I_n & -u \\ 0 & I_n\end{psmallmatrix}$ for all primes $q$. Thus we get $J^{\frac{1}{2}}(\alpha_u, z) = h(\xi, z) = h(\alpha_u, z) = p^{\frac{rn}{2}}$ by (\ref{h2}).

Finally, by Lemma 2.2 of \cite{Shimuratheta} we have
\[
	h\left(\begin{psmallmatrix} I_n & 0 \\ Y_0^2u & I_n\end{psmallmatrix}, \iota_{Y_0}z\right) = |-uz^{-1}+I_n|^{\frac{1}{2}}.
\]
Making use of $Y_r = Y_0p^r$ and combining all of the above, observe that both sides (\ref{factoridentity}) above coincide with $|Y_0i(z-u)|^{\frac{1}{2}}$. Thus the claim, and therefore the proposition, holds.
\end{proof}

As an example, assume that $\chi$ is primitive, that $\ell = \ell_{\chi}\geq 2^n-1$, and that $g = \theta_{\bar{\chi}}$. Let $H_{\chi}:=H_{\theta_{\bar{\chi}}}$ and $V_{\chi} := \Vol(\Gamma[\mfrak{b}^{-2}, \mfrak{b}^2\mfrak{y}_{\chi}]\back\mathbb{H}_n)$. Taking $r = \ell_{\chi}-2^n-1\geq 0$, we have $\mfrak{y}_r = \mfrak{y}_{\chi}$ and $\Gamma_r = \Gamma_{\chi}$, so applying the above proposition to the integral expression of (\ref{f0RS}) gives
\begin{align*}
	\begin{split}
		L_{\psi}(s, f_0, \bar{\chi}) &= \left[\Gamma_n\left(\tfrac{s-n-1+k+\mu}{2}\right)2c_{f_0}(\tau, 1)\right]^{-1}N(\mfrak{b})^{n(n+1)}|4\pi\tau|^{\frac{s-n-1+k+\mu}{2}}\\
	&\hspace{10pt}\times(-1)^{n[k]}\left(\frac{\Lambda_{\mfrak{c}_0}}{\Lambda_{\mfrak{y}_0}}\right)\left(\tfrac{2s-n}{4}\right)\prod_{q\in\mathbf{b}}g_q\left((\psi^{\mfrak{c}_0}\bar{\chi})(q)q^{-s}\right)\\
	&\hspace{10pt}\times \left\langle f_0, H_{\chi}|U_p^{\ell_{\chi}-2^n-1}W(Y_0)\right\rangle_{\mfrak{y}_0}V_{\chi}.
	\end{split}
\end{align*}

\section{A transformation formula of the theta series}
\label{thetasection}

\noindent Transformation formulae for theta series of the form $\theta_{\chi}|W(Y_{\chi})$ when $\chi$ is a primitive Dirichlet character are generally well-known entities. The precise formula of this section is encompassed by the generality of both Theorem A3.3 and Proposition A3.17 of \cite{Shimurabook}; what follows is a concrete derivation and calculation of the integrals found in the aforementioned results. Theorem A3.3 of \cite{Shimurabook} gives the existence of a $\mathbb{C}$-linear automorphism $\lambda\mapsto\action{\lambda}$ of $\mfrak{M}$ on the space of ``Schwartz functions on $M_n(\mathbb{Q}_{\mathbf{f}})$'', and it gives formulae of this action by $P_{\mathbb{A}}$ and the inversion $\iota = \begin{psmallmatrix} 0 & -I_n \\ I_n & 0\end{psmallmatrix}$. This is relevant since a more general class of theta series is defined using Schwartz functions $\lambda$ by
\begin{align*}
	\theta(z, \lambda) := \sum_{x\in M_n(\mathbb{Q})} \lambda(x_{\mathbf{f}})|x|^{\mu}e(x^T\tau xz),
\end{align*}
for a fixed $\tau\in S_+$ and $\mu\in\{0, 1\}$. If $\chi$ is a Hecke character of conductor $\mfrak{f}$, then putting $\lambda := \prod_p\lambda_p$ and 
\begin{align}
	\lambda_p(y) :=\begin{cases} 	
		1 &\text{if $y\in M_n(\mathbb{Z}_p)$ and $p\nmid\mfrak{f}$}, \\
		\chi_p(|y|) &\text{if $y\in GL_n(\mathbb{Z}_p)$ and $p\mid\mfrak{f}$}, \\
		0 &\text{otherwise},
	\end{cases} \label{schwartzchi}
\end{align}
gives the series $\theta(z, \lambda) = \theta_{\chi}^{(\mu)}(z; \tau)$ of (\ref{theta}).

Assume that $\chi$ is a Hecke character of conductor $p^{\ell_{\chi}}$ and let $\iota_{\chi} = \iota_{Y_{\chi}}$. Since $\iota_{\chi}\in C^{\theta}$, Proposition A3.17 of \cite{Shimurabook} says that 
\begin{align}
	\theta(z, \lambda)|W(Y_{\chi}) = \theta\left(z, \Waction{\lambda}\right)\label{thetatranss}
\end{align}
and so we calculate $\Waction{\lambda}$. Note $\iota_{\chi}^{-1} = \iota\sigma$ where
\[
	\sigma := \begin{pmatrix} -Y_{\chi}I_n & 0 \\ 0 & -Y_{\chi}^{-1}I_n\end{pmatrix}\in P,
\]
and so $\Waction{\lambda} = \iotaaction{(\action{\lambda})}$. Let $d = \frac{n^2}{2}$ if $n$ is even, $d = 0$ if $n$ is odd and let $d_py$ be the Haar measure on $M_n(\mathbb{Q}_p)$ such that the measure of $bM_n(\mathbb{Z}_p)$ is $|b|_p^{n^2/2}$ for any $b\in\mathbb{Q}$. Theorem A3.3 (5), and equation (A3.3) of \cite{Shimurabook}, and the definition of $\lambda$ in (\ref{schwartzchi}) above gives
\begin{align*}
 \iotaaction{(\action{\lambda})}_p(x) &= i^d\chi_{\infty}(-1)^n|Y_{\chi}|_p^{\frac{n^2}{2}}|\det(2\tau)|_p^{\frac{n}{2}}\int_{Y_{\chi}^{-1}GL_n(\mathbb{Z}_p)}\chi_p(|Y_{\chi}y|)e_p(-\tr(x^T2\tau y))d_{p}y \\
	&=i^d\chi_{\infty}(-1)^n|\det(2\tau)|_p^{\frac{n}{2}}\int_{GL_n(\mathbb{Z}_p)}\chi_p(|y|)e_p\left(-\tfrac{\tr(x^T2\tau y)}{Y_{\chi}}\right)d_py,
\end{align*}
making the change of variables $y\mapsto Y_{\chi}y$ in the last line. By the definition of $\hat{\tau}$ in (\ref{tauhat}) and $Y_{\chi} = N(\mfrak{tbc})p^{\ell_{\chi}}$ we have $\iotaaction{(\action{\lambda})}_p(x)$ is equal to
\begin{align*}
	\frac{|\det(2\tau)|_p^{\frac{n}{2}}}{(-i)^d\chi_{\infty}(-1)^n}\sum_{a\in GL_n(\mathbb{Z}/p^{\ell_{\chi}}\mathbb{Z})}\chi_p(|a|)e\left(\tfrac{\tr(x^T\hat{\tau}^{-1}a)}{N(\mfrak{bc})p^{\ell_{\chi}}}\right)\int_{p^{\ell_{\chi}}GL_n(\mathbb{Z}_p)} e_p\left(-\tfrac{\tr(x^T\hat{\tau}^{-1}y)}{N(\mfrak{bc})p^{\ell_{\chi}}}\right)d_py.
\end{align*}
The integral in the above equation is non-zero if and only if the integrand is a constant function in $y$ -- i.e. if and only if $x\in |N(\mfrak{bc})|_p^{-1}\hat{\tau}M_n(\mathbb{Z}_p)$ -- at which point it is $p^{-\ell_{\chi}\frac{n^2}{2}}$. Likewise by the same process, if $q\neq p$, we have $\iotaaction{(\action{\lambda})}_q(x)\neq 0$ if and only if $x\in |N(\mfrak{bc})|_q^{-1}\hat{\tau}M_n(\mathbb{Z}_q)$ at which point it is $|\det(2\tau)|_q^{\frac{n}{2}}$. Therefore $\iotaaction{(\action{\lambda})}(x)\neq 0$ if and only if $x\in N(\mfrak{bc})\hat{\tau}M_n(\mathbb{Z})$, for which
\begin{align}
	\iotaaction{(\action{\lambda})}(x) &= i^d\chi_{\infty}(-1)^n|2\tau|^{-\frac{n}{2}}p^{-\ell_{\chi}\frac{n^2}{2}}G_n(N(\mfrak{bc})^{-1}\hat{\tau}^{-1}x, \bar{\chi}), \label{actionlambda}
\end{align}
where, for any Hecke character $\ph$ of conductor $\mfrak{f}$ and $X\in M_n(\mathbb{Z})$, 
\[
	G_n(X, \ph) := \sum_{a\in GL_n(\mathbb{Z}/N(\mfrak{f})\mathbb{Z})}\ph_{\mfrak{f}}^{-1}(|a|)e\left(\tfrac{\tr(X^Ta)}{N(\mfrak{f})}\right)
\]
denotes the $n$-degree Gauss sum and put $G_n(\ph) = G_n(I_n, \ph)$. If $\ph$ is a primitive Dirichlet character then $G_n(X, \ph) = \ph(|X|)^{-1}G_n(\ph)$ if $(|X|, N(\mfrak{f}))=1$ and $G_n(X, \ph) = 0$ if $(|X|, N(\mfrak{f}))\neq 1$. So, under the assumption that $\chi$ is a primitive Dirichlet character and $x\in N(\mfrak{bc})\hat{\tau}M_n(\mathbb{Z})$, (\ref{actionlambda}) becomes
\begin{align}
	\iotaaction{(\action{\lambda})}(x) &= i^d\chi(-1)^n|2\tau|^{-\frac{n}{2}}p^{-\ell_{\chi}\frac{n^2}{2}}\chi(|N(\mfrak{bc})^{-1}\hat{\tau}^{-1}x|)G_n(\bar{\chi}). \label{actionlambda2}
\end{align}
Hence, by the calculation in (\ref{actionlambda2}), the transformation formula (\ref{thetatranss}) on theta series with Schwartz functions translates, when $\chi$ is a primitive Dirichlet character, to
\[
	\theta_{\chi}|W(Y_{\chi}) = \frac{i^d\chi(-1)^n}{|2\tau|^{\frac{n}{2}}}p^{-\ell_{\chi}\frac{n^2}{2}}G_n(\bar{\chi})\sum_{x\in N(\mfrak{bc})\hat{\tau}M_n(\mathbb{Z})} \chi(|N(\mfrak{bc})^{-1}\hat{\tau}^{-1}x|)|x|^{\mu}e(x^T\tau xz)
\]
and this becomes, by writing $x = N(\mfrak{bc})\hat{\tau}x'$ and $N(\sqrt{\mfrak{t}}) = |N(\mfrak{t})^{\frac{1}{2}}|$, the desired formula
\begin{align}
	\theta_{\chi}^{\mu}(z; \tau)|W(Y_{\chi}) = \chi(-1)^n\frac{i^dN(\mfrak{tbc})^{n\mu}}{|2\tau|^{\frac{n}{2}+\mu}}p^{-\ell_{\chi}\frac{n^2}{2}}G_n(\bar{\chi})\theta_{\bar{\chi}}^{(\mu)}\left(N(\sqrt{\mfrak{t}}\mfrak{bc})^2\frac{z}{2}; \hat{\tau}\right). \label{thetatrans}
\end{align}
	
\section{Fourier expansions of Eisenstein series}
\label{eisensteinsection}

\noindent The holomorphic projection map $\mathbf{Pr}:C_{\kappa}^{\infty}(\Gamma)\to\mathcal{M}_{\kappa}(\Gamma, \psi)$ and its explicit action on Fourier coefficients is well-known when $2n<\kappa\in\mathbb{Z}$ -- see Theorem 4.2 of \cite[p. 71]{Panchishkin}. This has a simple extension to the half-integral weight case with the formulae remaining unchanged, and we did this in \cite[Theorem 3.1]{Mercuri2}. 

Given Proposition \ref{tracereduce} and the transformation formula (\ref{thetatrans}), it will be germane to give the explicit Fourier development of $\mathbf{Pr}([\theta_{\chi}^{\star}\mathcal{E}^{\star}]|U_p^r)$, where 
\begin{align*}
	\theta_{\chi}^{\star}(z) :&= \theta_{\chi}^{(\mu)}\left(N(\sqrt{\mfrak{t}}\mfrak{bc})^2\frac{z}{2}; \hat{\tau}\right), \\
	\mathcal{E}^{\star}(z) :&= \mathcal{E}(z, \tfrac{2m-n}{4}; k-\tfrac{n}{2}-\mu, \bar{\eta}, \Gamma_r)|W(Y_r),
\end{align*}
 for certain values $m\in\frac{1}{2}\mathbb{Z}$ defined below. To ease up on notation, let
\[
	\delta := n\pmod{2}\in\{0, 1\}.
\]

The projection map is only applicable for certain values $s$ at which the Eisenstein series satisfies growth conditions; restriction to the set of special values, $\Omega_{n, k}$, at which the standard $L$-function satisfies algebraicity results guarantees this and this set is given by
\begin{align*}
	\Omega_{n, k}^+ :&= \left\{m\in\tfrac{1}{2}\mathbb{Z}\bigg|\tfrac{k-m-\mu}{2}\in\mathbb{Z}, n<m\leq k-\mu\right\}, \\
	\Omega_{n, k}^- :&= \left\{m\in\tfrac{1}{2}\mathbb{Z}\bigg|\tfrac{m+k-\mu-1}{2}\in\mathbb{Z}, 2n+1-k+\mu\leq m\leq n\right\}, \\
	\Omega_{n, k} :&= \Omega_{n, k}^-\cup\Omega_{n, k}^+.
\end{align*}

\begin{proposition}\label{eisenexp} For any $\varsigma\in S_+$, define
\[
	V_{\varsigma} := \left\{(\varsigma_1, \varsigma_2)\in M_n(\mathbb{Z})\times S_+\mid \tfrac{N(\sqrt{\mfrak{t}}\mfrak{bc})^2}{2}\varsigma_1^T\hat{\tau}\varsigma_1 + \varsigma_2 = \varsigma\right\}.
\]
Assume that $k>2n$, $\chi$ is a Dirichlet character, and $m\in\Omega_{n, k}$. For any $\beta\in\mathbb{Z}$, there exists a polynomial $P(\sigma, \sigma'; \beta)\in\mathbb{Q}[\varsigma_{ij}, \varsigma_{ij}'\mid 1\leq i\leq j\leq n]$, defined on $\sigma, \sigma'\in S_+$; a finite subset $\mathbf{c}$ of primes; polynomials $f_{\sigma, q}\in\mathbb{Z}[t]$, defined for each $\sigma\in S_+$ and $q\in\mathbf{c}$, whose coefficients are independent of $\chi$; and a factor
\begin{align*}
	C_{\pm}^{\star}(\sigma, m) :&= i^{-n\left(\left[k-\frac{n}{2}-\mu\right]\right)}N(\mfrak{b}^2\mfrak{y}_r)^{n\left(\frac{3n-2m}{2}-k+\mu\right)} 2^{n(k-\mu+\frac{3}{2})}\pi^{n\left(\frac{m+k-n-\mu}{2}\right)} \\
	&\hspace{10pt}\times\Gamma_n\left(\tfrac{m+k-n-\mu}{2}\right)^{-1}|\sigma|^{m_{\pm}}\prod_{q\in\mathbf{c}}f_{\sigma, q}(\bar{\eta}(q)q^{\frac{n+\delta-1}{2}-m}),
\end{align*}
where $m_+ = m-n-\frac{1}{2}$ and $m_- = 0$, such that if $m\in\Omega_{n, k}\back\{n+\frac{1}{2}\}$ (and $m\neq n+\frac{3}{2}$ if $n>1$ and $(\psi^*\chi)^2 = 1$), then $\mathbf{Pr}([\theta_{\chi}^{\star}\mathcal{E}^{\star}]|U_p^r)$ has non-zero Fourier coefficients only when $\sigma>0$ at which
\[
	c\left(\sigma, 1; \mathbf{Pr}([\theta_{\chi}^{\star}\mathcal{E}^{\star}]|U_p^r)\right) = \sum_{(\sigma_1, \sigma_2)\in V_{p^r\sigma}}\chi(|\sigma_1|)|\sigma_1|^{\mu}C^{\star}_+(\sigma_2, m)P\left(\sigma_2, p^r\sigma; \tfrac{k-m-\mu}{2}\right)
\]
if $m\in\Omega_{n, k}^+$ whereas
\[
	c\left(\sigma, 1; \mathbf{Pr}([\theta_{\chi}^{\star}\mathcal{E}^{\star}]|U_p^r)\right) = \sum_{(\sigma_1, \sigma_2)\in V_{p^r\sigma}}\chi(|\sigma_1|)|\sigma_1|^{\mu}C^{\star}_-(\sigma_2, m)P\left(\sigma_2, p^r\sigma; \tfrac{k+m-\mu-1-2n}{2}\right)
\]
if $m\in\Omega_{n, k}^-$.

Furthermore, the polynomial $P(\sigma, \sigma'; \beta)$ satisfies $P(\sigma, \sigma'; \beta) \equiv |\sigma|^{\beta}\pmod{\sigma_{ij}'}$.
\end{proposition}

When $k\in\mathbb{Z}$ and $n$ is even the above kind of result is well-known, see for example Theorem 4.6 of \cite[p. 77]{Panchishkin}. Since the definition of the projection map remains unchanged, we can obtain the above in a similar manner, by using results on the Fourier development of integral and half-integral weight Eisenstein series as follows.

Let $\kappa\in\frac{1}{2}\mathbb{Z}$ be such that $2\kappa + n\notin 2\mathbb{Z}$, and let $\mathcal{E}_{\kappa}(z, s) = \mathcal{E}(z, s; \kappa, \bar{\eta}, \Gamma[\mfrak{x}^{-1}, \mfrak{xy}])$, assuming as always that $(\mfrak{x}^{-1}, \mfrak{xy})\subseteq 2\mathbb{Z}\times 2\mathbb{Z}$ if $\kappa\notin\mathbb{Z}$. Further assume that $N(\mfrak{y})$ and $N(\mfrak{x})$ are both squares and let $Y := \sqrt{N(\mfrak{xy})}\in\mathbb{Z}$. If $2s+\frac{n}{2}\in\Omega_{n, k}^{\pm}$, $s\neq\frac{n+1}{4}$, and $s\neq \frac{n+3}{4}$ if $n>1$ and $\eta^2 = 1$, then by Proposition 17.6 of \cite{Shimurabook}, the analytic continuation of the Eisenstein series, and the fact that $\iota\Gamma[\mfrak{x}^{-1}, \mfrak{xy}]\iota = \Gamma[\mfrak{xy}, \mfrak{x}^{-1}]$, we have
\[
	\mathcal{E}_{\kappa}(z, s)|\iota =
	\displaystyle\sum_{\substack{0<\sigma\in S_+ \\ N(\mfrak{xy})\sigma\in S^{\triangledown}}} c^{\pm}(\sigma, y, s)e(\sigma x),
\]
where, if $\sigma>0$, we have by Propositions 16.9 and 16.10 of \cite{Shimurabook} that
\begin{align*}
	c^{\pm}(\sigma, y, s) :&= i^{n(\kappa-[\kappa])}N(\mfrak{xy})^{-\frac{n(n+1)}{2}}|y|^{s-\frac{\kappa}{2}}\xi(y, \sigma; s+\tfrac{\kappa}{2}, s-\tfrac{\kappa}{2}) \\
	&\hspace{10pt}\times\prod_{q\in\mathbf{c}}f_{Y^2\sigma, q}(\bar{\eta}(q)q^{-2s+[\kappa]-\kappa}), \\
	\xi(g, h; s, s') :&= \int_{S_{\infty}}e(-hx)|x-ig|^{-s}|x-ig|^{-s'}dx,
\end{align*}
defined for $0< g\in S_{\infty}, h\in S_{\infty}$, and $s, s'\in\mathbb{C}$. The above integral converges for large enough $\Re(s+s')$, but is continued analytically via the hyperconfluent geometric function $\omega(g, h; s, s')$ of \cite{Shimurahyp}. Through this analytic continuation one can represent $\xi(y, \sigma, s+\frac{\kappa}{2}, s-\frac{\kappa}{2})$, for the above values of $s$, in terms of the polynomial 
\begin{align*}
	R(g; \beta, s') :&= (-1)^{\beta n}e^{\tr(g)}|z|^{\beta+s'}\det\left[\frac{\partial_n}{\partial_ng}\right]^{\beta} (e^{-\tr(g)}|z|^{-s'}), \\
	\frac{\partial_n}{\partial_ng} :&= \left(\frac{1+\delta_{ij}}{2}\frac{\partial}{\partial g_{ij}}\right)_{i, j=1}^n, 
\end{align*}
where $0\leq\beta\in\mathbb{Z}$. This is obtained by using, in order, the relation (17.11) and analytic continuation of the hyperconfluent geometric function of \cite[Theorem 3.1]{Shimurahyp}; the properties (4.7.K) and (4.10) of \cite{Shimurahyp} and the definitions (3.23)--(3.24) of \cite[p. 63]{Panchishkin}; and, finally, Proposition 3.2 of \cite{Shimurahyp} to get
\begin{align}
	\begin{split}
	c^+(\sigma, y, s) &= C(\sigma, s)|4\pi y|^{s-\frac{\kappa}{2}}R\left(4\pi\sigma y; \tfrac{\kappa}{2}-s, \tfrac{n+1-\kappa}{2}-s\right)e^{-2\pi\tr(\sigma y)}, \\
	c^-(\sigma, y, s) &= C(\sigma, s)|4\pi y|^{\frac{n+1-\kappa}{2}-s}R\left(4\pi \sigma y; s+\tfrac{\kappa-n-1}{2}, s-\tfrac{\kappa}{2}\right)e^{-2\pi\tr(\sigma y)}, \\
	C(\sigma, s) :&= i^{-n[\kappa]}N(\mfrak{xy})^{-\frac{n(n+1)}{2}}2^{n(\kappa+\frac{n+3}{2})}\pi^{n(s+\frac{\kappa}{2})}\Gamma_n(s+\tfrac{\kappa}{2})^{-1} \\
	&\hspace{10pt}\times |\sigma|^{2s+\frac{\kappa-n-1}{2}}\prod_{q\in\mathbf{c}}f_{Y^2\sigma, q}(\bar{\eta}(q)q^{-2s+[\kappa]-\kappa}).
	\end{split} \label{cpm}
\end{align}
Now, since $\mathcal{E}_{\kappa}(z, s)|W(Y) = Y^{-n\kappa}(\mathcal{E}_{\kappa}(\cdot, s)|\iota)(Y^2z)$, we have that
\[
	\mathcal{E}_{\kappa}(z, s)|W(Y) = \sum_{0<\sigma\in S_+^{\triangledown}}c_{\mfrak{y}}^{\pm}(\sigma, y, s)e(\sigma x),
\]
where $c_{\mfrak{y}}^{\pm}(\sigma, y, s) := Y^{-n\kappa}c^{\pm}(Y^{-2}\sigma, Y^2y, s)$ are given explicitly by
\begin{align*}
	c_{\mfrak{y}}^+(\sigma, y, s) &= Y^{-n\kappa}C(Y^{-2}\sigma, s)|4\pi \sigma y|^{s-\frac{\kappa}{2}}R\left(4\pi\sigma y; \tfrac{\kappa}{2}-s, \tfrac{n+1-\kappa}{2}-s\right)e^{-2\pi\tr(\sigma y)}, \\
	c_{\mfrak{y}}^-(\sigma, y, s) &= Y^{-n\kappa}C(Y^{-2}\sigma, s)|4\pi\sigma y|^{\frac{n+1-\kappa}{2}-s}R\left(4\pi\sigma y;  s-\tfrac{n+1-\kappa}{2}, s-\tfrac{\kappa}{2}\right)e^{-2\pi\tr(\sigma y)}.
\end{align*}
Put $C_{\mfrak{y}}^+(\sigma, s) := Y^{-n\kappa}|\sigma|^{s-\frac{\kappa}{2}}C(Y^{-2}\sigma, s)$ and $C_{\mfrak{y}}^-(\sigma, s) := Y^{-n\kappa}|\sigma|^{\frac{n+1-\kappa}{2}-s}C(Y^{-2}\sigma, s)$.

Now let $g\in\mathcal{M}_{\ell}$ be a holomorphic modular form and let $F_g^{\star}(\cdot, s) = g[\mathcal{E}_{\kappa}(\cdot, s)|W(Y)]$. By analogy to Theorem 4.6 of \cite[p. 77]{Panchishkin} and using (\ref{cpm}), the coefficients after application of $\mathbf{Pr}$ are given by
\begin{align}
	c(\sigma, 1; \mathbf{Pr}(F_{g}^{\star}(\cdot, s)|U_p^r)) = \sum_{\sigma_1+\sigma_2=p^r\sigma}c_g(\sigma_1, 1)C_{\mfrak{y}}(\sigma, s)P\left(\sigma_1, p^r\sigma; \tfrac{\kappa}{2}-s\right) \label{projexpplus}
\end{align}
when $2s+\frac{n}{2}\in\Omega_{n, k}^+$, $s\neq \frac{n+1}{4}$ ($s\neq\frac{n+3}{4}$ if $\eta^2 = 1$ and $n > 1$), and $\sigma>0$, whereas
\begin{align}
	c(\sigma, 1; \mathbf{Pr}(F_g^{\star}(\cdot, z)|U_p^r)) = \sum_{\sigma_1+\sigma_2=p^r\sigma}c_g(\sigma_1, 1)C_{\mfrak{y}}(\sigma, s)P\left(\sigma_1, p^r\sigma; s+\tfrac{\kappa-n-1}{2}\right) \label{projexpminus}
\end{align}
if $2s+\frac{n}{2}\in\Omega_{n, k}^-$ and $\sigma>0$. In both cases the coefficients are zero for $\sigma\leq 0$.

Specialising (\ref{projexpplus}) and (\ref{projexpminus}) to the case $\mfrak{x} = \mfrak{b}^2$, $\mfrak{y} = \mfrak{y}_r$, $\kappa = k-\frac{n}{2}-\mu$, $\ell = \frac{n}{2}+\mu$, $g = \theta_{\chi}^{\star}$, and $s = \frac{2m-n}{4}$ for $m\in\Omega_{n, k}^{\pm}$, and also putting $C^{\star}_{\pm}(\sigma, m) := C_{\mfrak{y}_r}^{\pm}(\sigma, \frac{2m-n}{4})$ gives Proposition \ref{eisenexp}.

\section{$p$-adic interpolation}
\label{interpolation}

\subsection{$p$-adic measures and the main theorem}

\noindent Though complex $L$-functions are defined on variables $s\in\mathbb{C}$, they can equally be viewed as Mellin transforms of the continuous characters $\mathbb{R}_{>0}\to\mathbb{C}^{\times}; y\mapsto y^s$. In this latter vantage point, $p$-adic $L$-functions can naturally be constructed as Mellin transforms of continuous characters on $\mathbb{Z}_p^{\times}$ with respect to a $p$-adic measure. 

Fix a prime $p\nmid\mfrak{c}$, let $\mathbb{C}_p := \widehat{\overline{\mathbb{Q}}}_p$ denote the completion of the algebraic closure of $\mathbb{Q}_p$, and fix an embedding $\iota_p:\overline{\mathbb{Q}}\hookrightarrow\mathbb{C}_p$. The $p$-adic norm naturally extends to $\mathbb{C}_p$ and its ring of integers is given by
\[
	\mathcal{O}_p := \{x\in\mathbb{C}_p\mid |x|_p\leq 1\}.
\]
The domain of the $p$-adic $L$-function will be
\[
	X_p := \{x\in \Hom(\mathbb{Z}_p^{\times}, \mathbb{C}_p^{\times})\mid x\ \text{is continuous}\}.
\]
The discussion in \cite[pp. 23--25]{Panchishkin} concerning the decomposition of $X_p$ tells us that any $\mathbb{C}_p$-analytic function $F$ on $X_p$ is uniquely determined by its values $F(\chi_0 \chi)$ for a fixed $\chi_0\in X_p$ and $\chi$ ranging over non-trivial elements of $X_p^{\tors}$. This torsion subgroup can be identified as the group of primitive Dirichlet characters having $p$-power conductor. So to define a $p$-adic measure, it is enough to give its values on $\chi x_p^{[m]}$ where $\chi$ is a non-trivial primitive Dirichlet character of $p$-power conductor, $[m] = m-\frac{1}{2}\in\mathbb{Z}$, and 
\begin{align*}
	x_p^{[m]}:\mathbb{Z}_p^{\times}&\to\mathbb{C}_p^{\times} \\
	y&\mapsto y^{[m]}.
\end{align*}

\begin{definition}\label{distribution} Let $LC(\mathbb{Z}_p^{\times}, \mathbb{C}_p)$ denote the $\mathbb{C}_p$-module of all locally constant functions $\mathbb{Z}_p^{\times}\to\mathbb{C}_p$, and let $A$ be a $\mathbb{C}_p$-module. An $A$-valued distribution on $\mathbb{Z}_p^{\times}$ is an $A$-linear homomorphism
\[
	\nu:LC(\mathbb{Z}_p^{\times}, \mathbb{C}_p)\to A,
\]
which we denote by
\[
	\nu(\phi) = \int_{\mathbb{Z}_p^{\times}}\phi d\nu, 
\]
for any $\phi\in LC(\mathbb{Z}_p^{\times}, \mathbb{C}_p)$.

When $A =\mathbb{C}$ these are called \emph{complex distributions}, whereas when $A=\mathbb{C}_p$ they are \emph{$p$-adic distributions}.
\end{definition}

Since $\mathbb{Z}_p^{\times} = \varprojlim(\mathbb{Z}/p^i\mathbb{Z})^{\times}$ is a profinite group, taken with respect to the natural projections $\pi_{ij}:(\mathbb{Z}/p^i\mathbb{Z})^{\times}\to(\mathbb{Z}/p^j\mathbb{Z})^{\times}$ for each $i\geq j$, to any distribution there associates a system of functions $\nu_i:(\mathbb{Z}/p^i\mathbb{Z})^{\times}\to A$ satisfying
\[
	\nu_j(y) = \sum_{x\in \pi_{ij}^{-1}(y)}\nu_i(x), \hspace{20pt} y\in (\mathbb{Z}/p^j\mathbb{Z})^{\times}.
\]
This association works by noting that each $\phi\in LC(\mathbb{Z}_p^{\times}, \mathbb{C}_p)$ factors through some $(\mathbb{Z}/p^i\mathbb{Z})^{\times}$ and by 
\[
	\int_{\mathbb{Z}_p^{\times}}\phi d\nu = \sum_{x\in \mathbb{Z}/p^i\mathbb{Z}}\phi(x)\nu_i(x).
\]
The compatibility criterion of \cite[p. 17]{Panchishkin} tells us when we can run the above process backwards.

\begin{proposition}[(Compatibility criterion)]\label{compat} Consider and arbitrary system of functions $\{\nu_i:(\mathbb{Z}/p^i\mathbb{Z})^{\times}\to A\}_{i=1}^{\infty}$. If we have, for any fixed $j\in\mathbb{Z}$ and any function $\phi_j\in LC((\mathbb{Z}/p^j\mathbb{Z})^{\times}, A)$, that the sum
\[
	\sum_{y\in (\mathbb{Z}/p^i\mathbb{Z})^{\times}}\phi_j(\pi_{ij}(y))\nu_i(y)
\]
is independent of $i$ for large enough $i\geq j$, then there exists a distribution $\nu$ on $\mathbb{Z}_p^{\times}$ associated to $\{\nu_i\}_i$.
\end{proposition}

\begin{definition} \label{measure} Let $C(\mathbb{Z}_p^{\times}, \mathbb{C}_p)$ denote the topological $\mathbb{C}_p$-module of all continuous functions $\mathbb{Z}_p^{\times}\to\mathbb{C}_p$. A \emph{$p$-adic measure} is a $\mathbb{C}_p$-module homomorphism
\[
	\nu:C(\mathbb{Z}_p^{\times}, \mathbb{C}_p)\to\mathbb{C}_p.
\]
\end{definition}

So distributions are generally quite easy to define; $p$-adic measures arise from $p$-adic distributions that are $p$-adically bounded. Hence defining a distribution interpolating $L$-values is relatively trivial and showing that these expressions are bounded is the crux of the matter. To do this, we will invoke the abstract Kummer congruences, which criterion is well-known in generality and is due to Katz in \cite[p. 258]{Katz}; we give a specialisation of it.

\begin{proposition}[(Kummer Congruences)]\label{kummer} Suppose, for an index set $I$, that $\{f_i\}_{i\in I}\subseteq C(\mathbb{Z}_p^{\times}, \mathcal{O}_p)$ is such that $\spa_{\mathbb{C}_p}\{f_i\mid i\in I\}$ is dense in $C(\mathbb{Z}_p^{\times}, \mathbb{C}_p)$. For a given system $\{a_i\}_{i\in I}\subseteq\mathcal{O}_p$, there exists an $\mathcal{O}_p$-module homomorphism $\nu:C(\mathbb{Z}_p^{\times}, \mathcal{O}_p)\to\mathcal{O}_p$ such that 
\[
	\int_{\mathbb{Z}_p^{\times}}f_id\nu = a_i
\]
if and only if, for any finite subset $S\subseteq I$ and any system $\{b_i\}_{i\in S}\subseteq\mathbb{C}_p$, the condition
\[
	\sum_{i\in S}b_if_i\subseteq p^N\mathcal{O}_p
\]
for an integer $N$ implies that
\[
	\sum_{i\in S} b_ia_i\in p^N\mathcal{O}_p.
\]
\end{proposition}

The proof of this can be found in \cite[pp. 19--20]{Panchishkin}; it covers $\mathbb{C}_p$-valued measures as well by multiplication of some non-zero constant. An easy example of these criteria is the Fourier coefficients of the Eisenstein series given in the previous section. Recall the finite set of primes $\mathbf{c}$ and polynomials $f_{\sigma, q}\in\mathbb{Z}[t]$ from Proposition \ref{eisenexp}.

\begin{corollary} \label{Sigma} If $m-\frac{1}{2}\in\mathbb{Z}$ and $0<\sigma\in S_+$, then there exists a $p$-adic distribution $\Sigma_{\sigma, m}$ defined on non-trivial elements $\chi\in X_p^{\tors}$ by
\[
	\Sigma_{\sigma, m}(\chi) = \iota_p\left[\prod_{q\in\mathbf{c}}f_{\sigma, q}(\bar{\chi}(q)q^{\frac{n+\delta-1}{2}-m})\right].
\]
Setting $\Sigma_{\sigma} = \Sigma_{\sigma, \frac{1}{2}}$ defines a $p$-adic measure that satisfies
\[
	\int_{\mathbb{Z}_p^{\times}}\chi x_p^{[m]}d\Sigma_{\sigma} = \Sigma_{\sigma, m}(\chi).
\]
\end{corollary}

\begin{proof} That $\Sigma_{\sigma, m}$ satisfies the compatibility criterion is immediate. By taking $\Re(s)\to\infty$ in the identity (16.46) of \cite{Shimurabook} we see that the product of polynomials has no constant term. Take $\chi x_p^{[m]}$ as the system $\{f_i\}_{i\in I}$ in the statement of the Kummer congruences. If $\mathcal{X}\subseteq X_p^{\tors}$ is a finite subset and $\sum_{\chi\in\mathcal{X}}b_{\chi}\chi x_p^{[m]}\subseteq p^N\mathcal{O}_p$, then 
\[	
	\sum_{\chi\in\mathcal{X}}b_{\chi}\int_{\mathbb{Z}_p^{\times}}\chi x_p^{[m]}d\Sigma_{\sigma}\in p^N\mathcal{O}_p
\]
is immediate, by the crucial fact that the coefficients of $f_{\sigma, q}$ are independent of $\chi$.
\end{proof}

If $\nu(\chi x_p^{[m]}) = a_m(\chi)$ for some $a_m:\{\text{$\mathbb{T}$-valued characters}\}\to\mathbb{C}_p$ and $\omega$ is a primitive $\mathbb{T}$-valued character whose conductor is prime to $p$, then the twist of $\nu$ by $\omega$, given by $[\nu\otimes\omega](\chi x_p^{[m]}) := a_m(\chi\omega)$, is also a $p$-adic measure.

A non-zero Hecke eigenform $f$ with Satake $p$-parameters $(\lambda_{p, 1}, \dots, \lambda_{p, n})$ is \emph{$p$-ordinary} if $|\lambda_0|_p =1$, where recall that $\lambda_0 = p^{\frac{n(n+1)}{2}}\lambda_{p, 1}\cdots\lambda_{p, n}$. As usual, take a half-integral weight $k$, ideals $\mfrak{b}$ and $\mfrak{c}$ satisfying (\ref{BC1}) and (\ref{BC2}), a normalised Hecke character $\psi$ satisfying (\ref{PC1}) and (\ref{PC2}). The main theorem is given as follows.

\begin{theorem}\label{main} Let $p\nmid\mfrak{c}$ be a prime, $k>2n$, and $f\in\mathcal{S}_k(\Gamma[\mfrak{b}^{-1}, \mfrak{bc}], \psi)$ be a $p$-ordinary Hecke eigenform. Assume the existence of $\tau\in S_+$ such that $c_f(\tau, 1)\neq 0$, $c(\tau, 1; f_0)\neq 0$, and recall $\mfrak{t}$ as an integral ideal such that $h^T(2\tau)^{-1}h\in 4\mfrak{t}^{-1}$ for all $h\in\mathbb{Z}^n$. There exist bounded $\mathbb{C}_p$-analytic functions
\begin{align*}
	\nu_f^{\pm}:X_p\to\mathbb{C}_p
\end{align*}
that are uniquely determined by the following. In both cases $\chi\in X_p^{\tors}$ is a primitive Dirichlet character of conductor $p^{\ell_{\chi}}$ with $1\leq \ell_{\chi}\in\mathbb{Z}$; $\eta = \psi\bar{\chi}\rho_{\tau}$; $\mu\in\{0, 1\}$ is chosen so that $(\psi_{\infty}\chi)(-1) = (-1)^{[k]+\mu}$; recall $\hat{\tau} = N(\mfrak{t})(2\tau)^{-1}$; put $\Lambda_{\tau}(s) := (\Lambda_{\mfrak{c}}/\Lambda_{\mfrak{tc}})(\frac{2s-n}{4})$, which is a finite product of Euler factors defined by Section \ref{Lfunsection}, and also put $\mfrak{g}_{\tau}(s) :=~\prod_{q\in\mathbf{b}}g_q((\psi^{\mfrak{c}p}\bar{\chi})(q)q^{-s})^{-1}$, which is a product of polynomials in $\mathbb{Z}[t]$ also defined in Section \ref{Lfunsection}; and recall $d = \frac{n^2}{2}$ if $n$ is even, $d = 0$ if $n$ is odd.
\begin{enumerate}[(i)]
	\item For any $m-\frac{1}{2}\in\mathbb{Z}$ with $n\leq m\leq k-\mu$, the measure $\nu_f^+$ is given by
\begin{align*}
	\int_{\mathbb{Z}_p^{\times}} \chi x_p^{[m]} d\nu_f^+ &= \iota_p\left[\frac{(-1)^{n[k]}|2\tau|^{\frac{n}{2}+\mu}}{i^dN(\mfrak{tbc})^{n\mu}}\left|-\tfrac{N(\sqrt{\mfrak{t}}\mfrak{bc})^2}{2}\hat{\tau}\right|^{-\frac{k+m-\mu-1-2n}{2}}\right. \\
	&\hspace{50pt}\times\left.\frac{p^{n\ell_{\chi}(n+1-k-m)}G_n(\bar{\chi})}{\Lambda_{\tau}(m)\mfrak{g}_{\tau}(m)}\lambda_0^{-\ell_{\chi}}\frac{L_{\psi}(m, f, \bar{\chi})}{\pi^{n(k+m-n)}\langle f, f\rangle}\right],
\end{align*}
whenever $[m]\equiv [k] + \mu\pmod{2}$ (i.e. whenever $m\in\Omega_{n, k}^+$) and $m\neq n+\frac{1}{2}$ (with the further condition that $m\neq n+\frac{3}{2}$ if $(\psi^*\bar{\chi})^2 = 1$ and $n > 1$), otherwise the integral is zero.
	\item For any $m-\frac{1}{2}\in\mathbb{Z}$ with $2n+1-k+\mu\leq m\leq n$, the measure $\nu_f^-$ is given by
\begin{align*}
	\int_{\mathbb{Z}_p^{\times}} \chi x_p^{[m]} d\nu_f^+ &= \iota_p\left[\frac{(-1)^{n[k]}|2\tau|^{\frac{n}{2}+\mu}}{i^dN(\mfrak{tbc})^{n\mu}}\left|-\tfrac{N(\sqrt{\mfrak{t}}\mfrak{bc})^2}{2}\hat{\tau}\right|^{-\frac{k+3m-\mu-2-4n}{2}}\right. \\
	&\hspace{50pt}\times\left.\frac{p^{n\ell_{\chi}(n+1-k-m)}G_n(\bar{\chi})}{\Lambda_{\tau}(m)\mfrak{g}_{\tau}(m)}\lambda_0^{-\ell_{\chi}}\frac{L_{\psi}(m, f, \bar{\chi})}{\pi^{n(k+m-n)}\langle f, f\rangle}\right],
\end{align*}
whenever $[m]\equiv \mu + 1-[k]\pmod{2}$ (i.e. whenever $m\in\Omega_{n, k}^-$), otherwise the integral is zero.
\end{enumerate}
\end{theorem}

\begin{remark} The condition that $m\neq n+\frac{1}{2}$ (and $m\neq n+\frac{3}{2}$ when $(\psi^*\bar{\chi})^2 = 1$) arises as a result of complications in the Fourier expansion of the Eisenstein series at this value, as seen in the previous section. It is unique to the half-integral weight case since $n+\frac{1}{2}$ does not belong to the set of special values when $k\in\mathbb{Z}$. Most likely it can still be interpolated since one can use the Kubota-Leopoldt measure to interpolate the extra Fourier coefficients arising here, but it is not necessary in order to give the existence of the measure.
\end{remark}

The $p$-adic Mellin transform of a $p$-adic measure is defined by
\[
	L_{\nu}(x) := \int_{\mathbb{Z}_p^{\times}} xd\nu
\]
for any $x\in X_p$.

\begin{definition}\label{padiclfun} Let $f\in \mathcal{S}_k(\Gamma, \psi)$ be a $p$-ordinary Hecke eigenform. The $p$-adic $L$-functions of $f$ are defined by:
\[
	\mathcal{L}_p^{\pm}(s, f, \chi) := L_{\nu_f^{\pm}}(\chi x_p^{s-\frac{1}{2}}) = \int_{\mathbb{Z}_p^{\times}}\chi x_p^{s-\frac{1}{2}}d\nu_f^{\pm}.
\]
\end{definition}

\subsection{Proof of Theorem \ref{main}}

\noindent The proof of the main theorem now follows along the following lines: prove the existence of $p$-adic distributions interpolating $L_{\psi}(m, f, \bar{\chi})$ and show that they satisfy the Kummer congruences, thus defining $p$-adic measures. Though it is enough to define the $p$-adic distribution in terms of non-trivial primitive characters $\chi\in X_p^{\tors}$, to use the Kummer congruences we need \emph{all} characters in $X_p^{\tors}$ and we achieve this by lifting the undesirable primitive characters (i.e. the trivial character) into desirable imprimitive characters. The definition of the distribution on primitive characters is similar to that seen in Theorem \ref{main}. For imprimitive characters, we cannot use the transformation formula of the theta series (\ref{thetatrans}), instead we define it in terms of the Rankin-Selberg Dirichlet series
\[
	D(s, f, g) := \sum_{\sigma\in S_+/GL_n(\mathbb{Z})}\nu_{\sigma}^{-1}c_f(\sigma, 1)\overline{c_g(\sigma, 1)}|\sigma|^{-s-\frac{k-\ell}{2}},
\]
where $f\in\mathcal{M}_k$, $g\in\mathcal{M}_{\ell}$, and $\nu_{\sigma} := \#\{a\in GL_n(\mathbb{Z})\mid a^T\sigma a=\sigma\}$. This Rankin-Selberg Dirichlet series has an integral expression similar to (\ref{RS}) -- in fact it is used as an intermediary step in the proof of (\ref{RS}) -- and the flexibility in choice of $g$ allows us to pre-empt the right-hand side of the transformation formula (\ref{thetatrans}).

\begin{proposition}\label{distplus} There exists a complex distribution $\nu_s^+$ on $\mathbb{Z}_p^{\times}$ which is uniquely determined on Dirichlet characters of $p$-power conductor $p^{\ell_{\chi}}$ as follows. If $\chi$ is primitive and $1\leq \ell_{\chi}\in\mathbb{Z}$, then it is defined by
\begin{align}
	\begin{split}
	\nu_s^+(\chi) :&=\frac{(-1)^{n[k]}|2\tau|^{\frac{n}{2}+\mu}}{i^dN(\mfrak{tbc})^{n\mu}}\left|-\tfrac{N(\sqrt{\mfrak{t}}\mfrak{bc})^2}{2}\hat{\tau}\right|^{\frac{-k+s-\mu-1-2n}{2}}\\
	&\hspace{50pt}\times\frac{p^{-n\ell_{\chi}(n+1-k-s)}G_n(\bar{\chi})}{\Lambda_{\tau}(s)\mfrak{g}_{\tau}(s)}\lambda_0^{-\ell_{\chi}}L_{\psi}(s, f, \bar{\chi}),
	\end{split}
	\label{plusdistprim}
\end{align}
where $d$, $\Lambda_{\tau}$ and $\mfrak{g}_{\tau}$ are as in Theorem \ref{main}.

In general, for any $\ell>\ell_{\chi}\geq 0$, let $\chi_{\ell}$ denote the character modulo $p^{\ell}$ associated to $\chi$ and, for any $r>\ell$, define
\begin{align}
	\begin{split}
	\nu_s^+(\chi):&=\frac{|\tau|^{\frac{s-n-1+k+\mu}{2}}}{c_{f_0}(\tau, 1)}\left|-\tfrac{N(\sqrt{\mfrak{t}}\mfrak{bc})^2}{2}\hat{\tau}\right|^{-\frac{k+s-\mu-1-2n}{2}}\Lambda_{\mfrak{y}_0}(\tfrac{2s-n}{4})\\
	&\hspace{50pt}\times p^{rn(\frac{5n}{2}+2-2k-s)}\lambda_0^{-r}D\left(\tfrac{2s-3n-2}{4}, f_0, \theta_{\chi_{\ell}}^{\star}|W(Y_r)\right), 
	\end{split}\label{plusdistimp}
\end{align}
where, recall, $\theta_{\chi_{\ell}}^{\star}(z) = \theta_{\chi_{\ell}}^{(\mu)}(N(\sqrt{\mfrak{t}}\mfrak{bc})^2z/2; \hat{\tau})$.
\end{proposition}

\begin{proof} By the compatibility criterion, Proposition \ref{compat}, we just need to show that the definition of $\nu_s$ is independent of $\ell$ and $r$. When $\chi$ is primitive, this is immediate. The expression (\ref{plusdistimp}) is evidently independent of $\ell$ since $\ell>\ell_{\chi}$. Now fix $\ell$, to show independence of $r$ let $V(M)$ for $M\in\mathbb{Z}$ be the operator associated to $\begin{psmallmatrix} MI_n & 0 \\ 0 & M^{-1}I_n\end{psmallmatrix}$, which acts as $g|V(M) = M^{n\ell}g(M^2z)$ if $g$ is of weight $\ell$. Notice $W(Y_r) = W(Y_{\ell})V(p^{r-\ell})$ as operators, so the Dirichlet series $D(\frac{2s-3n-2}{4}, f_0, \theta_{\chi_{\ell}}^{\star}|W(Y_r))$ becomes
\[
	p^{n(\ell-r)(\frac{n}{2}+\mu)}\sum_{\sigma\in S_+/GL_n(\mathbb{Z})}\nu_{\sigma}^{-1}c_{f_0}(p^{2r-2\ell}\sigma, 1)\overline{c\left(\sigma, 1; \theta_{\chi_{\ell}}^{\star}|W(Y_{\ell})\right)}|p^{2(\ell-r)}\sigma|^{-\frac{s-n-1+k+\mu}{2}}.
\]
We have $c_{f_0}(p^{2r-2\ell}\sigma, 1) = p^{n(n+1-k)(\ell-r)}c(\sigma, 1; f_0|U_p^{r-\ell})$ and so the powers of $p^r$ in (\ref{plusdistimp}) cancel. Since $f_0|U_p = \lambda_0 f_0$, the proposition is proved.
\end{proof}

\begin{remark} Through the identities \cite[(5.9b, 8.8)]{Shimuraint} relating $D(s, f_0, g)$ to $L_{\psi}(s, f_0, \chi)$, Corollary \ref{pstablfun}, the transformation formula (\ref{thetatrans}) with the fact that $G_n(\chi)^{-1} = \chi(-1)^n p^{-n^2\ell_{\chi}}G_n(\bar{\chi})$, and the manipulations on $D(s, f, g)$ found in the above proof, one can check that the two definitions, (\ref{plusdistprim}) and (\ref{plusdistimp}), coincide if $\chi$ is primitive (i.e. when $\ell = \ell_{\chi}$).
\end{remark}

\begin{proposition} If $k>2n$ then, for any Dirichlet character $\chi$ of $p$-power conductor and $m\in\Omega_{n, k}^+$, we have
\[
	\frac{\nu_m^+(\chi)}{\pi^{n(k+m-n)}\langle f, f\rangle}\in\overline{\mathbb{Q}}.
\]
\end{proposition}

\begin{proof} In Theorem 7.6 of \cite{Mercuri2} we showed the existence of a non-zero constant $\mu(\Lambda, k, \psi)$ through which the Petersson inner product, and subsequently the $L$-value, satisfied an algebraicity result. Plugging $g = f$ into that theorem of \cite{Mercuri2} gives $\langle f, f\rangle \in\mu(\Lambda, k, \psi)\overline{\mathbb{Q}}$. So whenever $\chi$ is primitive, this is immediate from the main theorem, Theorem 7.8, of \cite{Mercuri2}. This is also given in \cite[Theorem 28.8]{Shimurabook}.

If $\chi$ is not primitive, then use the unfolded integral expression \cite[(8.5)]{Shimuraint} of $D(s, f, g)$ to obtain the expression
\begin{align}
	\begin{split}
		\nu_m^+(\chi) &= \left[2c_{f_0}(\tau, 1)\Gamma_n\left(\tfrac{m-n-1+k+\mu}{2}\right)\right]^{-1} \left|-\tfrac{N(\sqrt{\mfrak{t}}\mfrak{bc})^2}{2}\hat{\tau}\right|^{-\frac{k+m-\mu-1-2n}{2}}\\
	&\hspace{10pt}\times N(\mfrak{b})^{n(n+1)} |4\pi\tau|^{\frac{m-n-1+k+\mu}{2}}p^{rn(\frac{5n}{2}+2-2k-m)}\lambda_0^{-r} \\
	&\hspace{10pt}\times\left\langle f_0, \theta_{\chi_{\ell}}^{\star}|W(Y_r)\mathcal{E}(z, \tfrac{2m-n}{4}; k-\tfrac{n}{2}-\mu, \bar{\eta}, \Gamma_r)\right\rangle_{\mfrak{y}_r}V_r.
	\end{split}\label{distplusintexp}
\end{align}
Repeating the process of proving the algebraicity of $L$-values in \cite{Mercuri2} -- applying Proposition 7.5 with $g = \theta_{\chi_{\ell}}^{\star}|W(Y_r)$ and Theorem 7.6 found in that paper -- proves the proposition in this case too.
\end{proof}

By the above proposition we can define a $p$-adic distribution $\nu_m^{0+}$ for all $m-\frac{1}{2}\in\mathbb{Z}$ with $n\leq m\leq k-\mu$ by putting
\[
	\nu_m^{0+}(\chi) := \iota_p\left[\frac{\nu_m^+(\chi)}{\pi^{n(k+m-n)}\langle f, f\rangle}\right],
\]
for $m\in\Omega_{n, k}^+\back\{n+\frac{1}{2}\}$ (i.e. whenever $(\psi_{\infty}\chi)(-1) = (-1)^{[m]}$), and by otherwise putting $\nu_m^{0+}(\chi) = 0$ (and moreover $\nu_{n+\frac{3}{2}}^{0+}(\chi) = 0$ if $n>1$ and $(\psi^*\bar{\chi})^2 = 1$). 
To make the following expressions more manageable, we collect superfluous terms into a constant $C_r$, independent of $\chi$, as follows:
\begin{align*}
	C_{r} :&= (-1)^{n(\left[k-\frac{n}{2}-\mu\right])}\left[\pi^{n(k+m-n)}2c_{f_0}(\tau, 1)\Gamma_n\left(\tfrac{m+k+\mu-n-1}{2}\right)\right]^{-1}N(\mfrak{b})^{n(n+1)} \\
	&\hspace{10pt}\times|4\pi\tau|^{\frac{m+k+\mu-n-1}{2}} p^{rn(\frac{5n}{2} + 2-2k-m)}V_r.
\end{align*}
The factor of $(-1)$ appears as a result of $\theta_{\chi_{\ell}}^{\star}|W(Y_r)^2$ in the following calculation. Combining the integral expression (\ref{distplusintexp}) above with the case $g = \theta_{\chi_{\ell}}^{\star}|W(Y_r)$, for large enough $r$, of Proposition \ref{tracereduce}, we get
\begin{align}
	\nu_m^{0+}(\chi) = \iota_p\left[C_r\left|-\tfrac{N(\sqrt{\mfrak{t}}\mfrak{bc})^2}{2}\hat{\tau}\right|^{-\frac{k+m-\mu-1-2n}{2}}\lambda_0^{-r}\frac{\langle f_0, [\theta_{\chi_{\ell}}^{\star}\mathcal{E}^{\star}]|U_p^rW(Y_0)\rangle_{\mfrak{y}_0}}{\langle f, f\rangle}\right], \label{distplusreduce}
\end{align} 
where recall $\mathcal{E}^{\star}(z) = \mathcal{E}(z, \tfrac{2m-n}{4}; k-\tfrac{n}{2}-\mu, \bar{\eta}, \Gamma_r)|W(Y_r)$.
In light of the Fourier expansion in Proposition \ref{eisenexp} we make one final, artificial adjustment to this expression by inserting a constant $D_k$. Define
\begin{align*}
	D_k :&= \frac{i^{n(\left[k-\frac{n}{2}-\mu\right])}2^{-n(k-\mu+\frac{3}{2})}}{N(\mfrak{b}^2\mfrak{y}_r)^{n\left(\frac{3n-2m}{2}-k+\mu\right)}\pi^{\frac{n(m+k-n-\mu)}{2}}}\Gamma_n\left(\tfrac{m+k-n-\mu}{2}\right), \\
	\mfrak{R}_r^+(\cdot, m; \chi_{\ell}) :&= D_k\left|-\tfrac{N(\sqrt{\mfrak{t}}\mfrak{bc})^2}{2}\hat{\tau}\right|^{-\frac{k+m-\mu-1-2n}{2}}\mathbf{Pr}\left([\theta_{\chi_{\ell}}^{\star}\mathcal{E}^{\star}]|U_p^r\right),
\end{align*}
which latter is an element of $\mathcal{M}_k(\Gamma[\mfrak{b}^{-2}, \mfrak{b}^2\mfrak{y}_{\chi}], \psi)$ (which is true for all $m\in\Omega_{n, k}$). For the values of $m\in\Omega_{n, k}^+$ given in Proposition \ref{eisenexp} (which are those upon which our distribution $\nu_m^{0+}$ is non-zero), it has cyclotomic Fourier coefficients that are non-zero only when $\sigma>0$ at which point, by Proposition \ref{eisenexp}, they are
\begin{align}
	\begin{split}
	c(\sigma, 1; \mfrak{R}_r^+(\cdot, m;\chi_{\ell})) &= \sum_{(\sigma_1, \sigma_2)\in V_{p^r\sigma}}\chi_{\ell}(|\sigma_1|)|\sigma_1|^{\mu}\mfrak{C}_+^{\star}(\sigma_2, m)P\left(\sigma_2, p^r\sigma; \tfrac{k-m-\mu}{2}\right), \\
	\mfrak{C}_+^{\star}(\sigma_2, m) :&= \left|-\tfrac{N(\sqrt{\mfrak{t}}\mfrak{bc})^2}{2}\hat{\tau}\right|^{-\frac{k+m-\mu-1-2n}{2}}|\sigma_2|^{m-n-\frac{1}{2}}\\
	&\hspace{10pt}\times\prod_{q\in\mathbf{c}}f_{\sigma, q}(\bar{\eta}(q)q^{\frac{n+\delta-1}{2}-m}).
	\end{split}\label{Rplus}
\end{align}
Insertion of $D_k$ to the expression (\ref{distplusreduce}) above leaves us with
\begin{align}
	\nu_m^{0+}(\chi) = \iota_p\left[C_rD_k^{-1}\lambda_0^{-r}\frac{\langle f_0, \mfrak{R}_r^+(\cdot, m;\chi_{\ell})|W(Y_0)\rangle_{\mfrak{y}_0}}{\langle f, f\rangle}\right]. \label{artificialplus}
\end{align}

\paragraph{The distribution $\nu_s^-$.} To define $\nu_s^-$ we just replace $\left|-N(\sqrt{\mfrak{t}}\mfrak{bc})^2\hat{\tau}/2\right|^{-\frac{k+s-\mu-1-2n}{2}}$ in the definitions (\ref{plusdistprim}) and (\ref{plusdistimp}) with $\left|-N(\sqrt{\mfrak{t}}\mfrak{bc})^2\hat{\tau}/2\right|^{-\frac{k+3s-\mu-2-4n}{2}}$. As before, for $m-\frac{1}{2}\in\mathbb{Z}$ with $2n+1-k+\mu\leq m\leq n$ normalise this into the following $p$-adic distribution
\begin{align*}
	\nu_m^{0-}(\chi) := \iota_p\left[\frac{\nu_m^-(\chi)}{\pi^{n(k+m-n)}\langle f, f\rangle}\right]. 
\end{align*}
whenever $m\in\Omega_{n, k}^-$ and $\nu_m^{0-}(\chi) = 0$ otherwise.

Define $\mfrak{R}_r^{-} := \left|-N(\sqrt{\mfrak{t}}\mfrak{bc})^2\hat{\tau}/2\right|^{-(m-n-\frac{1}{2})}\mfrak{R}_r^+$, this has cyclotomic Fourier coefficients that are non-zero only when $\sigma>0$. In such a case they are given, when $m\in\Omega_{n, k}^-$, by
\begin{align*}
	c(\sigma, 1; \mfrak{R}_r^-(\cdot, m; \chi_{\ell})) &= \sum_{(\sigma_1, \sigma_2)\in V_{p^r\sigma}}\chi_{\ell}(|\sigma_1|)|\sigma_1|^{\mu}\mfrak{C}_-^{\star}(\sigma_2, m)P\left(\sigma_2, p^r\sigma; \tfrac{k+m-\mu-1-2n}{2}\right),
\end{align*}
where $\mfrak{C}_-^{\star}(\sigma_2, m) = \left|-N(\sqrt{\mfrak{t}}\mfrak{bc})^2\hat{\tau}/2\right|^{-(m-n-\frac{1}{2})}\mfrak{C}_+^{\star}(\sigma_2, m)$. We obtain the expression
\begin{align}
	\nu_m^{0-}(\chi) = \iota_p\left[C_rD_k^{-1}\lambda_0^{-r}\frac{\langle f_0, \mfrak{R}_r^-(\cdot, m;\chi_{\ell})|W(Y_0)\rangle_{\mfrak{y}_0}}{\langle f, f\rangle}\right]. \label{artificialminus}
\end{align}
\par\bigskip
Define the linear functional
\begin{align*}
	\ell_f:\mathcal{M}_k(\Gamma[\mfrak{b}^{-2}, \mfrak{b}^2\mfrak{y}_0], \psi) &\to\overline{\mathbb{Q}} \\
	g&\mapsto \frac{\langle f_0, g|W(Y_0)\rangle}{\langle f, f\rangle},
\end{align*}
and we have $\ell_f(g) \in\mathbb{Q}(f, g, \Lambda, \psi)$ where $\Lambda$ are the eigenvalues of $f$. Taking the similarly-defined linear functional $\mathcal{L}_f$, found in (3.51) of \cite[p. 109]{Panchishkin}, we have that $\ell_f \in \langle f_0, f_0\rangle\langle f, f\rangle^{-1}\mathcal{L}_f\overline{\mathbb{Q}}$. The functionals $\ell_f$ and $\mathcal{L}_f$ are equal up to some algebraic constant, of bounded $p$-adic norm, determined by the differences of the operator $W(Y_0)$ between this paper and \cite{Panchishkin}. So, by (3.52) of \cite[p. 109]{Panchishkin}, there exist positive definite $\sigma_1, \dots, \sigma_t\in S_+^{\triangledown}$ and $\beta_1, \dots, \beta_t\in \mathbb{Q}(f, \Lambda, \psi)$ satisfying
\begin{align}
	\ell_f(g) = \sum_{i=1}^t \beta_ic_g(\sigma_i, 1). \label{functionalsum}
\end{align}

For any subset $\mathcal{X}\subseteq X_p^{\tors}$ take the integers $\ell$ large enough so that all $\chi_{\ell}$ for $\chi\in\mathcal{X}$ are non-trivial and then take $r$ so that they are all defined modulo $p^r$. Thus the expressions (\ref{artificialplus}) and (\ref{artificialminus}) hold for all $\chi\in\mathcal{X}$. We have
\begin{align}
	\nu_m^{0\pm}(\chi) = \iota_p\left[C_rD_k^{-1}\lambda_0^{-r}\ell_f(\mfrak{R}_r^{\pm}(\cdot, m; \chi_{\ell}))\right].  \label{nufinalexp}
\end{align}
By assumption $|\lambda_0|_p = 1$ and, by definition, $C_r$ and $D_k$ are independent of $\chi$ and have bounded $p$-adic valuation. So whether $\nu_m^{0\pm}(\chi)$ defines a measure or not is directly dependent on the $p$-adic boundedness of $\ell_f(\mfrak{R}_r^{\pm})$. By the expressions (\ref{Rplus}) and the analogous one for $\mfrak{R}_r^-$, the $\mathbb{Q}_{ab}$-coefficients of $\mfrak{R}_r^{\pm}$ do not depend on the modulus $p^{\ell}$ of $\chi$, and therefore setting $c_{\sigma, r, m}^{\pm}(\chi) := \iota_p[c(\sigma, 1; \mfrak{R}_r^{\pm}(\cdot, m; \chi_{\ell})]$ for $m$ as in Proposition \ref{eisenexp} defines a $p$-adic distribution. By (\ref{functionalsum}), we see that $\ell_f(\mfrak{R}_r^{\pm})$ is bounded if $c_{\sigma, r, m}^{\pm}$ give rise to a $p$-adic measure.

Fix $\sigma$ and for any $(\sigma_1, \sigma_2)\in V_{p^r\sigma}$ fix $\sigma_1$; by definition of $c_{\sigma, r, m}^{\pm}$ we may assume $p\nmid|\sigma_1|$. By Proposition \ref{eisenexp} we have, for any $\beta\in\mathbb{Z}$, the congruences
\begin{align}
	|\sigma_2|&\equiv \left|-\tfrac{N(\sqrt{\mfrak{t}}\mfrak{bc})^2}{2}\hat{\tau}\right||\sigma_1|^2\pmod{p^r\mathcal{O}_p}, \label{congruence1}\\
	P(\sigma_2, p^r\sigma; \beta) &\equiv |\sigma_2|^{\beta}\equiv \left[\left|-\tfrac{N(\sqrt{\mfrak{t}}\mfrak{bc})^2}{2}\hat{\tau}\right||\sigma_1|^2\right]^{\beta}\pmod{p^r\mathcal{O}_p}. \label{congruence2}
\end{align}
So we have, by definition of $\mfrak{R}_r^{\pm}$ (see (\ref{Rplus})), and Corollary \ref{Sigma}, that
\begin{align}
	c_{\sigma, r, m}^{\pm}(\chi) \equiv\sum_{(\sigma_1, \sigma_2)\in V_{p^r\sigma}}\chi_{\ell}(|\sigma_1|)|\sigma_1|^{k+m-1-2n}\left[\Sigma_{\sigma_2}\otimes\omega_{\tau}\right](\chi_{\ell} x_p^{[m]}) \pmod{p^r\mathcal{O}_p}, \label{congruence3}
\end{align}
where $\omega_{\tau}$ is the primitive character associated to $\bar{\psi}^*\rho_{\tau}$. The conductor of $\omega_{\tau}$ is $t\mfrak{c}$, where $t$ is the conductor of $\rho_{\tau}$, and $t\mfrak{c}$ is prime to $p$ by the following argument. By assumption $p\nmid\mfrak{c}$ we just show $(t, p) =1$. Since $\mfrak{bc}\subseteq 2\mathbb{Z}$ and $\mfrak{c}\subseteq 4\mathbb{Z}$ we see $p\nmid\mfrak{b}^{-1}$ as well, so that $\rho_{\tau}(p) \neq 0$ if and only if $\left(\frac{|2N(\mfrak{b})^{-1}\tau|}{p}\right)\neq 0$. That latter Legendre symbol makes sense since we know $2N(\mfrak{b}^{-1})\tau\in M_n(\mathbb{Z})$ by the Fourier coefficient property (\ref{c1}) of $f$. We can assume $|\sigma_2|\not\equiv 0\pmod{p^r\mathcal{O}_p}$, since otherwise $c_{\sigma, r, m}^{\pm}(\chi)\equiv 0\pmod{p^r}$ by congruence in (\ref{congruence2}). By definition, following by the use of the congruence in (\ref{congruence1}) above, we see
\begin{align*}
	\left(\frac{|2N(\mfrak{b})^{-1}\tau|}{p}\right)^r &= \left(\frac{N(\mfrak{b}^{-1}\mfrak{t})}{p}\right)^{rn} \left(\frac{|\hat{\tau}|}{p}\right)^r \\
&= \left(\frac{-N(\mfrak{bc^2})/2}{p}\right)^{rn}\left(\frac{|\sigma_2|}{p^r}\right),
\end{align*}
which is non-zero by assumption. So we see $p\nmid t$ and hence $\Sigma_{\sigma_2}\otimes\omega_{\tau}$ defines a $p$-adic measure. Now define $c_{\sigma, r}^{\pm} := c_{\sigma, r, \frac{1}{2}}^{\pm}$, which satisfies $c_{\sigma, r}^{\pm}(\chi x_p^{[m]}) = c_{\sigma, r, m}^{\pm}(\chi)$ by (\ref{congruence3}).

The Kummer congruences now complete the proof that $c_{\sigma, r}^{\pm}$ defines a measure. Assume 
\[
	\sum_{\chi\in\mathcal{X}}b_{\chi}\chi x_p^m\subseteq p^N\mathcal{O}_p,
\]
for some $b_{\chi}\in\mathcal{O}_p$. Then the congruence (\ref{congruence3}) above gives
\begin{align*}
\sum_{\chi\in\mathcal{X}}b_{\chi}c_{\sigma, r}^{\pm}(\chi x_p^{[m]})\equiv\sum_{(\sigma_1, \sigma_2)\in V_{p^r\sigma}}|\sigma_1|^{k-1-2n}\sum_{\chi\in\mathcal{X}}b_{\chi}\int_{\mathbb{Z}_p^{\times}}(\chi_{\ell}x_p^{[m]})(y|\sigma_1|)[d\Sigma_{\sigma_2}\otimes\omega_{\tau}](y)
\end{align*}
taken modulo $p^N\mathcal{O}_p$. The right-hand side of the above is clearly in $p^N\mathcal{O}_p$ since we have $|\sigma_1|\in\mathcal{O}_p^{\times}$ and we know $\Sigma_{\sigma_2}\otimes \omega_{\tau}$ is a measure satisfying the Kummer congruences.

We have shown that $\nu_m^{0\pm}$ defines a $p$-adic measure. To finish the proof of Theorem \ref{main} put $\nu_f^{\pm} := \nu_{\frac{1}{2}}^{0\pm}$. By (\ref{functionalsum}), (\ref{nufinalexp}), (\ref{congruence3}), and the fact that $c_{\sigma, r}^{\pm}(\chi x_p^{[m]}) = c_{\sigma, r, m}^{\pm}(\chi)$, we have 
\[
	\nu_f^{\pm}(\chi x_p^{[m]}) = \nu_m^{0\pm}(\chi). 
\]
The main identities of Theorem \ref{main} follow by using the definitions of the underlying distributions, for example (\ref{plusdistprim}), and the subsequent normalisations.

\affiliationone{
	S. Mercuri \\
	Department of Mathematical Sciences \\
	Durham University \\
	Lower Mountjoy \\
	Stockton Road \\
	Durham DH1 3LE \\
	United Kingdom
	\email{smmercuri@gmail.com}
}


\begin{thebibliography}{99}
\bibitem{Andrianov} 
 A.~Andrianov, 
`The Multiplicative Arithmetic or Siegel Modular Forms', 
{\em Russian Mathematical Surveys}, \textbf{34} (1) (1979), 75--148.

\bibitem{Bocherer}
S.~B\"{o}cherer \& C.G.~Schmidt, 
`$p$-adic measures attached to Siegel modular forms',
{\em Annales de l'Institut Fourier}, \textbf{50} (5) (2000), 1375--1443.

\bibitem{BouganisHerm}
T.~Bouganis,
`$p$-adic measures for Hermitian modular forms and the Rankin-Selberg method' in `Elliptic Curves, Modular Forms and Iwasawa Theory',
{\em Soringer Proceedings in Mathematics \& Statistics}, \textbf{188} (2016), 33--86.

\bibitem{Katz}
N.M.~Katz,
`$p$-adic $L$-functions for CM fields',
{\em Inventiones Mathematicae}, \textbf{48} (1978), 199--297.

\bibitem{Mercuri}
S.~Mercuri,
`The $p$-adic $L$-function for half-integral weight modular forms', 
to appear in {\em manuscripta mathematica}, (2018).

\bibitem{Mercuri2}
S.~Mercuri,
`Algebraicity of Metaplectic $L$-functions',
{\tt arXiv:1811.09155 [math.NT]}, (2018).

\bibitem{Panchishkin} 
A.~Panchishkin, 
`Non-Archimedean L-Functions of Siegel and Hilbert Modular Forms', 
{\em Lecture Notes in Mathematics}, \textbf{1471}, Springer-Verlag, 
Berlin-Heidelberg (1991).

\bibitem{Skinner1}
C.~Skinner,
`Main Conjectures and Modular Forms', 
{\em Current Developments in Mathematics}, \textbf{2004} (2006), 141--161.

\bibitem{Skinner}
C.~Skinner,
`The Iwasawa Main Conjectures for $GL_2$', 
{\em Inventiones Mathematicae}, \textbf{195} (1) (2014), 1 -- 277.

\bibitem{Shimurahyp}
G.~Shimura,
`Confluent hypergeometric functions on tube domains',
\emph{Annals of Mathematics}, \textbf{260} (1982), 269 -- 302.

\bibitem{Shimuraold} 
 G.~Shimura, 
`On Eisenstein Series of Half-integral Weight', 
{\em Duke Mathematical Journal}, \textbf{52} (1985), 281--314.

\bibitem{Shimuratheta}
G.~Shimura,
`On the transformation formulas of theta series',
{\em American Journal of Mathematics}, \textbf{115} (5) (1993), 1011--1052.

\bibitem{Shimuraint} 
 G.~Shimura, 
`Euler Products and Fourier Coefficients of Automorphic Forms on 
Symplectic Groups', 
{\em Inventiones Mathematicae}, \textbf{116} (1994), 531--576.

\bibitem{Shimurahalf} 
 G.~Shimura, 
`Zeta Functions and Eisenstein Series on Metaplectic Groups', 
{\em Inventiones Mathematicae}, \textbf{121} (1995), 21--60.

\bibitem{Shimuraexp} 
 G.~Shimura, 
`Convergence of Zeta Functions on Symplectic and Metaplectic Groups', 
{\em Duke Mathematical Journal}, \textbf{82} (2) (1996), 327--347.

\bibitem{Shimurabook} 
G.~Shimura, 
`Arithmeticity of Automorphic Forms', 
{\em Mathematical Surveys and Monographs}, \textbf{82}, Amer. Math. Soc. (2000).

\bibitem{Wan}
X.~Wan, 
`The Iwasawa Main Conjecture for Hilbert Modular Forms',
{\em Forum of Mathematics, Sigma}, \textbf{3} E18, doi:10.1017/fms.2015.16.

\bibitem{Weissmann}
M.~Weissmann,
`$L$-groups and parameters for covering groups',
{\em Ast\'{e}risque}, \textbf{398} (2018), 33--186.
\end{thebibliography}
\end{document}